\documentclass[a4paper,leqno,10pt]{amsart}

\usepackage{epsf,graphicx,epsfig}
\usepackage{amscd}
\usepackage{amsmath,latexsym,amssymb,amsthm}
\usepackage[nospace,noadjust]{cite}
\usepackage{textcomp}
\usepackage{setspace,cite}
\usepackage{lscape,fancyhdr,fancybox}
\usepackage{stmaryrd}
\usepackage[all,cmtip]{xy}
\usepackage{tikz}
\usepackage{cancel}
\usetikzlibrary{shapes,arrows,decorations.markings}
\usepackage{graphicx}

\usepackage{color}
\usepackage{url}
\usepackage{enumerate}
\usepackage[mathscr]{euscript}

  \theoremstyle{plain}

    \newtheorem{thm}{Theorem}[section]
    \newtheorem{prop}[thm]{Proposition}
     
   \newtheorem{lemma}[thm]{Lemma}

    \newtheorem{subsec}[thm]{}
\theoremstyle{definition}
    \newtheorem{defn}[thm]{Definition}
        \newtheorem{remark}[thm]{Remark}
\theoremstyle{remark}

\title{}
\author{}
\date{}
\usepackage{amssymb}

\usepackage{hyperref}
\hypersetup{
	colorlinks,
	citecolor=blue,
	filecolor=black,
	linkcolor=blue,
	urlcolor=black
}

\begin{document}
\title{Deformations of Loday-type algebras and their morphisms}

\author{Apurba Das}
\address{Department of Mathematics and Statistics,
Indian Institute of Technology, Kanpur 208016, Uttar Pradesh, India.}
\email{apurbadas348@gmail.com}



\noindent

\thispagestyle{empty}

\begin{abstract}
We study formal deformations of multiplication in an operad. This closely resembles Gerstenhaber's deformation theory for associative algebras. However, this applies to various algebras of Loday-type and their twisted analogs. We explicitly describe the cohomology of these algebras with coefficients in a representation. Finally, deformation of morphisms between algebras of the same Loday-type is also considered.
\end{abstract}

\maketitle

\section{Introduction}

Algebraic deformation theory first appeared in a pioneer work of M. Gerstenhaber in 1964 \cite{gers-def}. In his paper, he studies formal one-parameter deformations of associative algebras and shows that they are closely related to the Hochschild cohomology of associative algebras.
Since then, formal deformation theory has been studied for various other types of algebras, including Lie algebras, Leibniz algebras, and dialgebras \cite{ nij-rich , bala , maj-mukh }. 
Later on, Gerstenhaber and Schack also developed a deformation theory of associative algebra morphisms \cite{gers-sch} (see also \cite{fre-mar-yau}).

\medskip

Motivated by his deformation theory of associative algebras, Gerstenhaber defined certain operations on the Hochschild cochain complex of associative algebras which induce a rich structure on the cohomology \cite{gers}. This structure is now known as Gerstenhaber algebra. In \cite{gers-voro} the authors shed new light on the Gerstenhaber algebra structure on the Hochschild cohomology which led them to relate it with the Deligne's conjecture. Given a non-symmetric operad $\mathcal{O}$ with a multiplication $\pi \in \mathcal{O}(2)$, they showed that the cochain complex induced from the multiplication $\pi$ inherits a homotopy $G$-algebra structure. Hence the cohomology inherits a Gerstenhaber algebra structure. When one considers the endomorphism operad $\text{End}(A)$ associated to a vector space $A$, a multiplication on $\text{End}(A)$ is precisely an associative algebra structure on $A$ and the corresponding cochain complex is precisely the Hochschild cochain complex. Hence one recovers the result of \cite{gers}. It is important to remark that in the Hochschild cochain complex, there is a degree $-1$ pre-Lie product (which induces the degree $-1$ graded Lie bracket on the associated Gerstenhaber structure on Hochschild cohomology) that plays a pivotal role in the study of finite order deformations of the algebra.

\medskip

In this paper, we mainly concentrate in a class of algebras giving rise to a non-symmetric operad with a multiplication. We call them algebras of ``Loday-type''. Associative algebras belong to this class. In \cite{loday} Loday introduced a notion of (associative) dialgebra which gives rise to a Leibniz algebra in the same way an associative algebra gives rise to a Lie algebra. The Koszul dual operad of dialgebra is given by the operad of dendriform algebra. Dendriform algebras can be thought of as splitting of associative algebras and arise from Rota-Baxter operators of weight $0$ \cite{aguiar}. These two algebras are surprisingly related to some combinatorial objects, namely with planar binary trees. Later on, Loday and Ronco defined other algebras (associative trialgebras, dendriform trialgebras) which are related to planar trees (not necessarily binary) in the same way dialgebras and dendriform algebras are related to planar binary trees \cite{loday-ronco}. Dendriform trialgebras are splitting of associative algebras by three operations. In \cite{aguiar-loday} Aguiar and Loday introduced another class of algebras, called quadri-algebras, which are splitting of dendriform algebras. In the same spirit, Leroux \cite{leroux} defined a splitting of dendriform trialgebras and called them ennea-algebras. These algebras arose from infinitesimal bialgebras and commuting Rota-Baxter operators. It has been shown by Majumdar and Mukherjee \cite{maj-mukh2} that dialgebras are of Loday-type, i.e. dialgebras can be described by an operad with multiplication. The case of associative trialgebras is very similar. In \cite{yau} Yau has defined the same for dendriform algebras and dendriform trialgebras, however, there is some inaccuracy in the constructions of operads. It has been clarified for dendriform algebras in \cite{das4}. We show that dendriform trialgebras, quadri-algebras, and ennea-algebras are also of Loday-type. Note that Lie algebras, Leibniz algebras or algebras whose defining identity/identities has shufflings cannot be described by a non-symmetric operad with multiplication, hence they are not of Loday-type.




\medskip

Our aim in this paper is to study formal deformations of algebras that are of Loday-type. For that, in Section \ref{sec-3}, we first define formal deformations of multiplication in a non-symmetric operad. They are governed by the cohomology induced from the multiplication. The set of equivalence classes of formal deformations of multiplication is shown to be the moduli space of solutions of the Maurer-Cartan equation in a suitable dgLa. We also obtain an explicit deformation formula of a multiplication. Deformations of dialgebras and dendriform algebras as studied in  \cite{maj-mukh, das4} can be seen as deformations of the corresponding multiplications. By definition, deformations of other Loday-type algebras (dendriform trialgebras, quadri-algebras, ennea-algebras) as mentioned in the previous paragraph are given by deformations of the corresponding multiplications. One can also apply this method to a certain twisted analog of Loday-type algebras. It is important to remark that deformations of algebras over quadratic operads have been carried out in \cite{bala}. If $\mathcal{P}$ is a quadratic operad and $A$ is a $\mathcal{P}$-algebra, then deformation of $A$ as a $\mathcal{P}$-algebra can be viewed as a deformation of a certain Maurer-Cartan element in a gLa. This is governed by the operadic cohomology $H^\bullet_\mathcal{P}(A)$ of $A$ with coefficients in itself. However, to obtain the operadic cohomology $H^\bullet_\mathcal{P}(A)$, one needs to know very explicitly the operad $\mathcal{P}$ and its dual operad $\mathcal{P}^!$ (see  \cite{bala}). This is not always an easy task. For instance, the operad of quadri-algebras and its dual operad is hard to work \cite{aguiar-loday}. From this point of view, our cohomology and deformations are more elementary for Loday-type algebras. See Section \ref{sec-4} for the comparison between these two deformations and Section \ref{sec-5} for the comparison between the operadic cohomology and the cohomology induced from the multiplication for Loday-type algebras.

\medskip

Motivated from the fact that Loday-type algebras can be described by an operad with multiplication, in Section \ref{sec-5}, we explicitly define cohomology of Loday-type algebras with coefficients in a representation. As mentioned before, the comparison between this cohomology and operadic cohomology will be given (Remark \ref{comp-two-cohomo}). In the case of a dialgebra, our cohomology coincides with that of Frabetti \cite{frab}. We also show that the second cohomology group can be interpreted as equivalence classes of abelian extensions in the category of algebras of the same Loday-type.

\medskip

Finally, we also study deformations of morphisms between algebras of the same Loday-type. Let $A$ and $B$ be two algebras of same Loday-type. A morphism $f : A \rightarrow B$ between them makes $B$ into a representation of $A$ via $f$. We study deformations of $f$ by deforming the domain and codomain of $f$ as well. In the particular case of a dialgebra morphism, we get the deformation studied in \cite{yau2}.

\medskip

All vector spaces, linear maps, and tensor products are over a field $\mathbb{K}$ of characteristic $0$.

\section{Operads with multiplication}\label{sec-2}
In this section, we recall some basics on non-symmetric operads equipped with a multiplication. See \cite{gers-voro} for more details.

\begin{defn}
A non-symmetric operad (non-$\sum$ operad in short) in the category of vector spaces is a collection $\mathcal{O} = \{ \mathcal{O}(n)|~ n \geq 1 \}$ of vector spaces together with compositions
\begin{align}\label{operad-composition}
\gamma : \mathcal{O}(k) \otimes \mathcal{O}(n_1) \otimes \cdots \otimes \mathcal{O}(n_k) \rightarrow \mathcal{O}(n_1 + \cdots + n_k), ~~
f \otimes g_1 \otimes \cdots \otimes g_k \mapsto \gamma (f ; g_1, \ldots, g_k) 
\end{align}
which are associative in the sense that
\begin{align*}
\gamma \big(&  \gamma (f; g_1 , \ldots, g_k); h_1 , \ldots, h_{n_1 + \cdots + n_k}    \big) \\
&= \gamma \big( f; ~\gamma (g_1 ; h_1, \ldots, h_{n_1}), ~ \gamma ( g_2 ; h_{n_1 + 1}, \ldots, h_{n_1 + n_2}) , \ldots, ~\gamma (g_k ; h_{n_1 + \cdots + n_{k-1}+1}, \ldots, h_{n_1 + \cdots + n_k})  \big)
\end{align*}     
and there is an identity element id$ \in \mathcal{O} (1)$ such that
$\gamma (f ; \underbrace{\text{id}, \ldots, \text{id}}_{k \text{ times}} ) = f = \gamma (\text{id}; f)$, for  $f \in \mathcal{O} (k)$. 
\end{defn}

Alternatively, a non-symmetric operad can also be described by partial compositions
\begin{align*}
\circ_i : \mathcal{O}(m) \otimes \mathcal{O}(n) \rightarrow \mathcal{O}(m+n-1), \quad 1 \leq i \leq m
\end{align*}
satisfying
$$ \begin{cases} (f \circ_i g) \circ_{i+j-1} h = f \circ_i (g \circ_j h),  &\mbox{~ for } 1 \leq i \leq m, ~1 \leq j \leq n, \\ (f \circ_i g) \circ_{j+n-1} h = (f \circ_j h) \circ_i g,   & \mbox{~ for } 1 \leq i < j \leq m, \end{cases}$$
for $f \in \mathcal{O}(m), ~ g \in \mathcal{O}(n), ~ h \in \mathcal{O}(p),$
and an identity element $\text{id} \in \mathcal{O}(1)$ satisfying $ f \circ_i \text{id} = f =\text{id}\circ_1 f,$ for all $f \in \mathcal{O}(m)$ and $1 \leq i \leq m$. The two definitions of non-symmetric operad are related by
\begin{align}
f \circ_i g =~& \gamma (f ;~ \overbrace{\text{id}, \ldots, \text{id} , \underbrace{g}_{i\text{-th place}}, \text{id}, \ldots, \text{id}}^{m\text{-tuple}}),~~~ \text{ for }f \in \mathcal{O}(m), \label{ass-partial}\\
\gamma (f ; g_1, \ldots, g_k) =~&   (\cdots ((f \circ_k g_k) \circ_{k-1} g_{k-1}) \cdots ) \circ_1 g_1 , ~~~ \text{ for }f \in \mathcal{O}(k). \label{partial-ass}
\end{align}

A non-symmetric operad as above may be denoted by $(\mathcal{O}, \gamma, \text{id})$ or $(\mathcal{O}, \circ_i, \text{id})$.
A toy example of a non-symmetric operad is given by the endomorphism operad $\mathcal{O} = \text{End} (A)$ associated to a vector space $A$. For any $n \geq 1$, we define $\mathcal{O}(n) = \text{End} (A^{\otimes n}, A)$. The compositions (\ref{operad-composition}) are substitution of the values of $k$ operations in a $k$-ary operation as inputs. The identity element is given by the identity map on $A$. From now on, by an operad, we shall mean a non-symmetric operad. However, there is a notion of symmetric operad in which there is a right action of $\mathbb{K}[\mathbb{S}_n]$ on $\mathcal{O}(n)$ compatible with the partial compositions \cite{lod-val-book}.

Let $(\mathcal{O}, \gamma, \text{id})$ be an operad. If $f \in \mathcal{O}(n)$, we write $|f| = n-1$. In \cite{getz} Getzler  and Jones has defined the following brace operations
\begin{align}\label{gers-voro-brace}
\{ f \} \{ g_1, \ldots, g_n\} := \sum (-1)^\epsilon ~ \gamma (f ; \text{id}, \ldots, \text{id}, g_1, \text{id}, \ldots, \text{id}, g_n, \text{id}, \ldots, \text{id}),
\end{align} 
where the summation runs over all possible substitutions of $g_1, \ldots, g_n$ into $f$ in the prescribed order and $\epsilon = \sum_{p=1}^n |g_p| i_p$ with $i_p$ being the total number of inputs in front of $g_p$. We denote the circle product $\circ : \mathcal{O}(m) \otimes \mathcal{O}(n) \rightarrow \mathcal{O}(m+n-1)$ by
\begin{align*}
f \circ g := \{ f \} \{ g \} = \sum_{i=1}^m (-1)^{(i-1)|g|}~ f \circ_i g, ~~ \text{ for } f \in \mathcal{O}(m).
\end{align*}
The braces (\ref{gers-voro-brace}) satisfy certain pre-Jacobi identities, which in particular implies that the circle product $\circ$ satisfies the pre-Lie identities
\begin{align}\label{pre-lie-iden}
(f \circ g) \circ h - f \circ (g \circ h) = (-1)^{|g||h|} ((f \circ h) \circ g - f \circ (h \circ g)).
\end{align}
Hence the bracket
$[f, g ] := f \circ g - (-1)^{|f||g|} g \circ f$
defines a degree $-1$ graded Lie bracket on $\oplus_{n \geq 1} \mathcal{O}(n)$.

\begin{defn}
A multiplication on an operad $(\mathcal{O}, \gamma, \text{id})$ is an element $\pi \in \mathcal{O}(2)$ satisfying $\pi \circ \pi = 0$, or, equivalently, $\pi \circ_1 \pi = \pi \circ_2 \pi.$
\end{defn}

Let $(A, \mu)$ be an associative algebra. Then $\mu$ defines a multiplication on the endomorphism operad associated with $A$. In fact, the associativity of $\mu$ is equivalent to $\mu \circ_1 \mu = \mu \circ_2 \mu$ in the endomorphism operad. Thus, one might expect that some of the classical results for associative algebras can be extended to any operads equipped with a multiplication.

If $\pi$ is a multiplication on an operad $\mathcal{O}$, then the product
\begin{align*}
f \cdot g := (-1)^{|f| + 1}~ \{\pi\} \{f, g \}
\end{align*}
and the differential $d_\pi : \mathcal{O}(n) \rightarrow \mathcal{O}(n+1)$, $f \mapsto \pi \circ f - (-1)^{|f|} f \circ \pi$, makes the graded space $\oplus_{n \geq 1} \mathcal{O}(n)$ into a differential graded associative algebra. Thus the product passes to the cohomology $H^\bullet (\mathcal{O}, d_\pi)$. Moreover, it can be shown that the degree $-1$ graded Lie bracket $[~,~]$ also passes to the cohomology. Finally, the induced product and bracket on the cohomology $H^\bullet (\mathcal{O}, d_\pi)$ satisfy the graded Leibniz rule to become a Gerstenhaber algebra \cite{gers-voro}.

The above idea applies to the Hochschild cochain complex of associative algebras, the cochain complex of several other algebras including dialgebras, various other Loday-type algebras, some hom-type algebras and also applicable to singular cochain complex of topological spaces \cite{gers-voro, maj-mukh2, yau, das2, das4}. Therefore, the cohomology of these algebras inherits a Gerstenhaber structure.

\section{Deformations of multiplications}\label{sec-3}
In this section, we define formal deformations of multiplication in an operad. This is similar to the formal deformation theory of associative algebras developed by Gerstenhaber \cite{gers-def}.
The equivalence classes of deformations are described by the moduli spaces of solutions of the Maurer-Cartan equation in a certain dgLa. 

\subsection{Deformation}

Let $(\mathcal{O}, \gamma, \text{id})$ be an operad. Consider the space $\mathcal{O}(n) [[t ]]$ of formal power series in a variable $t$ with coefficients in $\mathcal{O}(n)$. One can linearly extend the circle products (or brace operations) to $\oplus_{n \geq 1} \mathcal{O}(n)[[t]].$

Let $\pi$ be a fixed multiplication on $\mathcal{O}$.
\begin{defn}
A formal $1$-parameter deformation of $\pi$ is given by a formal sum 
\begin{center}
$\pi_t = \pi_0 + \pi_1 t + \pi_2 t^2 + \cdots \in \mathcal{O}(2)[[t]] ~\text{ with } \pi_0 = \pi,$
\end{center}
satisfying $\pi_t \circ \pi_t = 0$. This is equivalent to a system of equations:
\begin{align}\label{deformation-eqn}
\sum_{i+j = n} \pi_i \circ \pi_j = 0,  \text{ for } n \geq 0.
\end{align}
\end{defn}

For $n = 0$, we have $\pi \circ \pi = 0$ which automatically holds from the assumption. For $n = 1$, we have $\pi \circ \pi_1 + \pi_1 \circ \pi = 0$, which implies that $d_\pi (\pi_1) = 0$. Thus, $\pi_1$ defines a $2$-cocycle in $(\mathcal{O}, d_\pi ).$
The $2$-cocycle $\pi_1$ is called the infinitesimal of the deformation. More generally, if $\pi_1 = \cdots = \pi_{n-1} = 0$, then $\pi_n$ is a $2$-cocycle. It is called the $n$-th infinitesimal of the deformation.

\begin{defn}\label{equi-defrm}
Two deformations $\pi_t = \sum_{i \geq 0} \pi_i~ t^i$ and $\pi'_t = \sum_{i \geq 0} \pi'_i~ t^i$
of $\pi$ are said to be equivalent if there exists a formal sum $\phi_t = \phi_0 + \phi_1 t + \phi_2 t^2 + \cdots \in \mathcal{O}(1)[[t]]$ (with $\phi_0 = \text{id} \in \mathcal{O}(1)$) such that
\begin{center}
$\phi_t \circ \pi'_t = \{\pi_t\} \{\phi_t, \phi_t \}.$
\end{center}
\end{defn}

This condition again leads to a system of equations:
\begin{align}
 \sum_{i+j = n} \phi_i \circ \pi'_j = \sum_{i+j+k = n} \{\pi_i\}\{\phi_j, \phi_k\}, ~~ \text{ for } n \geq 0.
\end{align}

For $n = 0$, the relation holds automatically as $\phi_0 = \text{id}$. However, for $n = 1$, it gives
\begin{align*}
\pi'_1 + \phi_1 \circ \pi = \{\pi \} \{ \text{id}, \phi_1 \}  + \{ \pi \} \{ \phi_1, \text{id} \} + \{ \pi_1 \} \{ \text{id}, \text{id} \},
\end{align*}
equivalently,
$\pi'_1 - \pi_1 = \pi \circ \phi_1 - \phi_1 \circ \pi = d_\pi (\phi_1).$
This shows that the infinitesimals corresponding to equivalent deformations are cohomologous and hence they correspond to the same cohomology class in $H^2 (\mathcal{O}, d_\pi).$\\

An infinitesimal deformation is a formal deformation modulo $t^2$. Thus, an infinitesimal deformation of $\pi$ is given by a sum $\pi_t = \pi + \pi_1 t$ satisfying $\pi_t \circ \pi_t \equiv 0$ (mod $t^2$). As before, $\pi_1$ defines a $2$-cocycle in $(\mathcal{O}, d_\pi)$ and equivalent infinitesimal deformations give rise to the same cohomology class in $H^2 (\mathcal{O}, d_\pi).$ Moreover, we have the following characterization of infinitesimal deformations.

\begin{prop}
There is a one-to-one correspondence between the space of equivalence classes of infinitesimal deformations of $\pi$ and the second cohomology $H^2 (\mathcal{O}, d_\pi).$
\end{prop}

\begin{proof}
Any $2$-cocycle $\pi_1 \in \mathcal{O}(2)$ defines an infinitesimal deformation given by $\pi_t = \pi + \pi_1 t$. For any cohomologous $2$-cocycle $\pi_1 + d_\pi (\phi_1) = \pi_1 + \pi \circ \phi_1 - \phi_1 \circ \pi$, for some $\phi_1 \in \mathcal{O}(1)$, the corresponding infinitesimal deformation is $\pi_t' = \pi + (\pi_1 + \pi \circ \phi_1 - \phi_1 \circ \pi) t$. It is easy to see that the sum $\phi_t = \text{id} + \phi_1 t$ satisfies
\begin{align*}
\phi_t \circ \pi_t' = \{ \pi_t \} \{ \phi_t, \phi_t \} ~~~~ (\text{mod } t^2).
\end{align*}
This shows that the infinitesimal deformations $\pi_t$ and $\pi_t'$ are equivalent.
\end{proof}

We now return to formal deformations.
A deformation $\pi_t$ of $\pi$ is said to be trivial if it is equivalent to the deformation $\pi_t' = \pi$.
A multiplication $\pi$ is called rigid if any deformation of $\pi$ is equivalent to a trivial deformation.

\begin{prop}
Let $\pi_t = \sum_{i \geq 0} \pi_i~ t^i$ be a nontrivial deformation of $\pi$. Then $\pi_t$ is equivalent to a deformation $\pi_t' = \pi + \sum_{i \geq p } \pi_i'~ t^i$, where the first nonzero term $\pi_p'$ is a $2$-cocycle but not a coboundary.
\end{prop}

\begin{proof}
Let $\pi_t = \sum_{i \geq 0} \pi_i~ t^i$ be a deformation such that $\pi_1= \cdots = \pi_{n-1} = 0$ and $\pi_n$ is the first nonzero term. Then it has been shown that $\pi_n$ is a $2$-cocycle. If $\pi_n$ is not a coboundary then we are done. If $\pi_n$ is a coboundary, that is, $\pi_n = - d_{\pi} (\phi_n)$, for some $\phi_n \in  \mathcal{O}(1)$, set
$ \phi_t = \text{id} + \phi_n t^n \in \mathcal{O}(1)[[t]]$.
We define $\pi'_t = \phi_t^{-1} \circ \{ \pi_t \} \{ \phi_t, \phi_t \}$. Then $\pi'_t$ defines a formal deformation of the form
\begin{center}
$\pi'_t = \pi +  \pi'_{n+1} ~t^{n+1} +  \pi'_{n+2}~ t^{n+2} + \cdots.$
\end{center}
Thus, it follows that $\pi'_{n+1}$ is a $2$-cocycle. If this $2$-cocycle is not a coboundary then we are done. Otherwise, we apply the same method again. In this way, we can get a required deformation.
\end{proof}

As a corollary, we obtain the following.

\begin{thm}
If $H^2 (\mathcal{O}, d_\pi) = 0$ then the multiplication $\pi$ is rigid.
\end{thm}

\medskip

Let $\pi$ be a multiplication on an operad $(\mathcal{O}, \gamma, \text{id})$. A finite sum $\pi_t = \sum_{i = 0}^n \pi_i t^i$ with $\pi_0 = \pi$ is said to be a deformation of order $n$ if it satisfies $\pi_t \circ \pi_t = 0~(\mathrm{mod ~} t^{n+1})$. In the following, we assume that $H^2 (\mathcal{O}, d_\pi) \neq 0$ so that one may obtain nontrivial deformations. Here, we consider the problem of extending a deformation of order $n$ to a deformation of order $n + 1$. Suppose there is an element $\pi_{n+1} \in \mathcal{O}(2)$ so that
$\overline{\pi}_t = \pi_t + \pi_{n+1}~ t^{n+1}$
is a deformation of order $n+1$. Then we say that $\pi_t$ extends to a deformation of order $n+1$.

Since we assume that $\pi_t =  \sum_{i = 0}^n  \pi_i t^i$ is a deformation of order $n$, it follows from (\ref{deformation-eqn}) that
\begin{align}\label{deform-rel}
\pi \circ \pi_i + \pi_1 \circ \pi_{i-1} + \cdots + \pi_{i-1} \circ \pi_1 + \pi_i \circ \pi = 0, ~~~ \text{ for } i = 1, 2, \ldots, n,
\end{align}
or, equivalently, $- d_\pi (\pi_i) = \sum_{p+q=i, p, q \geq 1} \pi_p \circ \pi_q$, for $i = 1, 2, \ldots, n.$
For $\overline{\pi}_t = \pi_t +  \pi_{n+1}~ t^{n+1}$ to be a deformation of order $n+1$, one more deformation equation needs to be satisfied
\begin{align}\label{alt-def-eqn}
- d_\pi (\pi_{n+1}) =   \sum_{i+j = n+1, i,j \geq 1} \pi_i \circ \pi_j . 
\end{align}
The right hand side of the above equation is called the obstruction to extend the deformation $\pi_t$ to a deformation of order $n+1$.

\begin{prop}
The obstruction is a $3$-cocycle, that is,
\begin{align*}
d_\pi \big(  \sum_{i+j=n+1, i, j \geq 1} \pi_i \circ \pi_j \big) = 0.
\end{align*}
\end{prop}
\begin{proof}
For any $f, g \in \mathcal{O}(2)$, it is easy to see that
\begin{align*}
d_\pi (f \circ g) =  f \circ d_\pi (g) - d_\pi (f) \circ g ~+~ g \cdot f -~ f \cdot g.
\end{align*}
(See \cite[Theorem 3]{gers-def} for the case of associative algebra.) Therefore,
\begin{align*}
d_{\pi} \big( \sum_{i+j =n+1, i, j \geq 1} \pi_i \circ \pi_j \big) =~&  \sum_{i+j =n+1, i, j \geq 1} \big( \pi_i \circ d_{\pi} (\pi_j) - d_{\pi} (\pi_i) \circ \pi_j \big) \\
=~& - \sum_{p+q+r = n+1, p, q, r \geq 1} \big(  \pi_p \circ (\pi_q \circ \pi_r) - (\pi_p \circ \pi_q) \circ \pi_r \big)\\ =~& - \sum_{p+q+r = n+1, p, q, r \geq 1} A_{p, q, r}    \qquad \text{(say)}.
\end{align*}
The product $\circ$ is not associative, however, it satisfies the pre-Lie identities (\ref{pre-lie-iden}). This in particular implies that $A_{p, q, r} = 0$ whenever $q = r$. Finally, if $q \neq r$ then $A_{p, q, r} + A_{p, r, q} = 0$ by the same identity (\ref{pre-lie-iden}). Hence we have $\sum_{p+q+r = n+1, p, q, r \geq 1} A_{p, q, r}   = 0$.
\end{proof}

It follows from the above proposition that the obstruction defines a cohomology class in $H^3 (\mathcal{O}, d_\pi)$. If this cohomology class is zero, then the obstruction is given by a coboundary (say $- d_{\pi} (\pi_{n+1})$). In other words, $\overline{\pi_t} = \pi_t +  \pi_{n+1}~ t^{n+1}$ defines a deformation of order $n+1$.

As a summary, we get the following.
\begin{thm}\label{3-zero-extension}
If $H^3 (\mathcal{O}, d_\pi) = 0$, every finite order deformation of $\pi$ can be extended to a deformation of next order.
\end{thm}

\subsection{Deformation space}
In this subsection, we describe the equivalence classes of formal deformations of $\pi$ as the solutions of the Maurer-Cartan equation in a dgLa. See \cite{dou-mar-zima} for the case of  associative algebra deformations.

We start with the following notations.
Let $\mathfrak{g} = \oplus_n \mathfrak{g}^n$ be a dgLa. Consider the new dgLa $L = \mathfrak{g} \otimes (t),$
where $(t) \subset \mathbb{K} [[t]]$ is the ideal generated by $t$. Therefore, degree $n$ elements of $L$ are of the form 
$\gamma = f_{1}t + f_{2}t^{2} + \cdots,$ where~each $f_{i} \in \mathfrak{g}^n$. 
The dgLa structure on $L$ is induced from the dgLa structure on $\mathfrak{g}$. An element $\gamma = f_{1}t+f_{2}t^{2}+ \cdots \in L^1  = \mathfrak{g}^1 \otimes (t)$ is Maurer-Cartan if  it satisfies
$$d \gamma + \frac{1}{2} [\gamma, \gamma ] = 0 \Leftrightarrow df_k + \frac{1}{2} \sum_{i+j=k} [f_i , f_j] = 0 ,~~\text{ for all }~k \geq 1.$$
Denote by MC($\mathfrak{g}$) the set of Maurer-Cartan elements in $L$.
Moreover, the gauge group of $\mathfrak{g}$, defined as
\begin{align*}
G(\mathfrak{g}) = \mathrm{exp} (L^0),
\end{align*}
where $\mathrm{exp} (L^0)$ denotes the group whose underlying space is $L^0 = \mathfrak{g}^0 \otimes (t)$ and the multiplication given by the Baker-Campbell-Hausdorff
formula (induced from the Lie algebra structure on $L^0$). The gauge group $G(\mathfrak{g})$ acts on $L^1 = \mathfrak{g}^1 \otimes (t)$ by
\begin{align*}
x \cdot \gamma  = \text{exp}(x) \cdot \gamma \cdot \text{exp}(-x) , \text{ for } x \in G(\mathfrak{g}), ~ \gamma \in L^1. 
\end{align*}
The gauge group preserves the space MC($\mathfrak{g}$) of Maurer-Cartan elements 
(see \cite{dou-mar-zima} for details). The quotient space
\begin{align*}
\mathfrak{Def(g)} = \text{MC}(\mathfrak{g}) / \text{G}(\mathfrak{g})
\end{align*}
is called the moduli space of solutions of the Maurer-Cartan equation in $L = \mathfrak{g} \otimes (t).$

\medskip

Let $\pi$ be a multiplication in an operad $\mathcal{O}$ and consider the dgLa $\mathfrak{g} = (\mathcal{O} (\bullet +1), [~,~], d_\pi)$. Then
\begin{align*}
\text{MC} (\mathfrak{g}) =~& \{ \gamma = \pi_1 t + \pi_2 t^2 + \cdots |~ \pi_i \in \mathcal{O}(2) \text{ and } d_\pi (\pi_k) + \frac{1}{2} \sum_{i+j=k} [\pi_i , \pi_j] = 0 ,~~\forall~k \geq 1 \}\\
=~& \{  \gamma = \pi_1 t + \pi_2 t^2 + \cdots |~ \pi_i \in \mathcal{O}(2) \text{ and } - d_\pi (\pi_k) = \sum_{i+j=k, i, j \geq 1} \pi_i \circ \pi_j , ~ \forall ~k \geq 1 \}.
\end{align*}
Therefore, $\gamma \in \mathrm{MC}(\mathfrak{g})$ if and only if $\pi_t = \pi + \gamma$ is a formal deformation of $\pi.$ Hence $\mathrm{MC}(\mathfrak{g})$ can be thought of as the set of all formal deformations of $\pi$. Observe that, in this example, the group $G(\mathfrak{g})$ is isomorphic to the group
\begin{align*}
H = \{ \phi_t = \text{id} + \phi_1 t + \phi_2 t^2 + \cdots |~ \phi_i \in \mathcal{O}(1) \}
\end{align*}
via $\text{exp} : G (\mathfrak{g}) \rightarrow H$ and the inverse is given by $\text{log} : H \rightarrow G (\mathfrak{g})$. Note that two formal deformations $\pi_t = \pi + \pi_1t + \pi_2 t^2 + \cdots $ and $\pi_t' = \pi + \pi_1' t + \pi_2' t^2 + \cdots$ are equivalent (i.e define same element in $\mathfrak{Def(g)}$) if and only if
\begin{align*}
\text{exp} (x) \cdot (\pi+ \pi_1t + \pi_2 t^2 + \cdots ) = (\pi + \pi_1' t + \pi_2' t^2 + \cdots) (\text{exp}(x) \otimes \text{exp}(x)), ~~ \text{ for some } x \in G (\mathfrak{g}).
\end{align*}
Hence we conclude that two formal deformations are equivalent in the sense of Definition \ref{equi-defrm} if and only if they lie in the same orbit of the gauge group action. In other words, we have the following.

\begin{prop}
Let $(\mathcal{O}, \pi)$ be an operad with a multiplication. The equivalence classes of formal deformations of $\pi$ is given by the moduli space $\mathfrak{Def(g)}$ of solutions of the Maurer-Cartan equation, where $\mathfrak{g}$ is the dgLa $(\mathcal{O} (\bullet +1), [~,~], d_\pi)$.
\end{prop}

\subsection{A deformation formula}

In this subsection, we give a deformation formula of a multiplication. When one considers the deformation of an associative algebra (i.e. deformation of multiplication in the endomorphism operad), one recovers the deformation constructed by Gerstenhaber \cite[Lemma 1]{gers3}. We start with the following lemma whose proof is similar to \cite[Lemma 1]{coll-gers-giaq}.

\begin{lemma}\label{deri-rule}
Let $D \in \mathcal{O}(1)$ be such that $d_\pi (D) = 0$ (i.e. $D$ is a $1$-cocycle). Then for any $\overline{D} \in \mathcal{O}(1)$,
\begin{align*}
D^p \circ (D^q \cdot \overline{D}^q) =~& \sum_{j=0}^p {p \choose j}~ D^{q +j} \cdot (D^{p-j} \circ \overline{D}^q), \\
D^p \circ (\overline{D}^q \cdot D^q) =~& \sum_{j=0}^p {p \choose j} ~ (D^j \circ \overline{D}^q) \cdot D^{q+p-j}.
\end{align*}
\end{lemma}

\begin{thm}\label{univ-formula-thm}
Let $D , \overline{D} \in \mathcal{O}(1)$ be such that $d_\pi (D) = d_\pi (\overline{D}) = 0$ (i.e. $D$ and $\overline{D}$ are both $1$-cocycles). Further, if $D \circ \overline{D} = \overline{D} \circ D$ then
\begin{align}\label{first-univ-formula}
\pi_t = \sum_{n = 0}^{\infty} \frac{t^n}{ n !} ~\{ \pi \} \{ D^n, \overline{D}^n \} = - \sum_{n = 0}^{\infty} \frac{t^n}{ n !}~ D^n \cdot \overline{D}^n
\end{align}
defines a deformation of $\pi$.
\end{thm}

\begin{proof}
Here $\pi_i = - \frac{1}{i!} D^i \cdot \overline{D}^i$. To prove that $\pi_t$ defines a deformation of $\pi$, one has to verify relations (\ref{deformation-eqn}). First observe that
\begin{align}
\sum_{i+j = n} \pi_i \circ_1 \pi_j = \sum_{i=0}^n \pi_i \circ_1 \pi_{n-i} =~&  \sum_{i=0}^n \frac{1}{i! (n-i)!} ~ (D^i \circ (D^{n-i} \cdot \overline{D}^{n-i})) \cdot \overline{D}^i \nonumber \\
=~& \sum_{i=0}^n \frac{1}{i! (n-i)!} \sum_{j=0}^i {i \choose j}~ D^{n-i+j} \cdot (D^{i-j} \circ \overline{D}^{n-i}) \cdot \overline{D}^i  \nonumber \\
=~& \sum_{i,j=0, 1, \ldots, n, i \geq j} \frac{1}{(n-i)! j! (i-j)!}~ D^{n-j} \cdot (D^j \circ \overline{D}^{n-i}) \cdot \overline{D}^i. \label{univ-formula}
\end{align}
Similarly,
\begin{align*}
\sum_{i+j = n} \pi_i \circ_2 \pi_j = \sum_{i=0}^n \pi_i \circ_2 \pi_{n-i} =~&  \sum_{i=0}^n \frac{1}{i! (n-i)!}~ D^i  \cdot (\overline{D}^i \circ (D^{n-i} \cdot \overline{D}^{n-i})) \\
=~& \sum_{i=0}^n \frac{1}{i! (n-i)!} \sum_{j=0}^i {i \choose j} ~D^i \cdot (\overline{D}^j \circ D^{n-i}) \cdot \overline{D}^{n-j} \\
=~& \sum_{i,j=0, 1, \ldots, n, i \geq j} \frac{1}{(n-i)! j! (i-j)!} ~D^i \cdot (\overline{D}^j \circ D^{n-i}) \cdot \overline{D}^{n-j}.
\end{align*}
By replacing the dummy variables $i \leftrightarrow n-j$ and $j \leftrightarrow n-i$ and using the fact that $D , \overline{D}$ commute, we get the same expression as in  (\ref{univ-formula}). Thus we obtain $\sum_{i+j = n} \pi_i \circ_1 \pi_j = \sum_{i+j = n} \pi_i \circ_2 \pi_j$ which is equivalent to $\sum_{i+j=n} \pi_i \circ \pi_j = 0$. Hence the proof.
\end{proof}

In terms of the dgLa $\mathfrak{g} = ( \mathcal{O}(\bullet +1), [~,~], d_\pi)$, the above deformation of $\pi$ is given by the Maurer-Cartan element
\begin{align*}
\gamma = - \sum_{n=1}^\infty \frac{t^n}{n !}~ D^n \cdot {\overline{D}}^n \in \mathfrak{g}^1 \otimes (t) = \mathcal{O}(2) \otimes (t).
\end{align*}

\section{Deformations of Loday-type algebras}\label{sec-4}
It is shown in \cite{maj-mukh2, das4} that dialgebra and dendriform algebra structure on a vector space can be seen as multiplication in suitable operads. It is very easy to verify that deformations of these algebras as developed in \cite{maj-mukh, das4} is equivalent to the deformation of the corresponding multiplications. We show that other Loday-type algebras (e.g. dendriform trialgebras, quadri-algebras, ennea-algebras) can also be described by multiplication in certain operads, and by definition, their deformation is given by deformation of the respective multiplications.

Let $\mathcal{P}$ be a non-symmetric quadratic operad with the (quadratic) dual operad $\mathcal{P}^!$. Then for any vector space $A$, the collection of spaces
\begin{align*}
\mathcal{O}(n) := \mathrm{Hom}_\mathbb{K}( (\mathcal{P}^!)^\vee (n) \otimes A^{\otimes n} , A), ~~~ \text{ for } n \geq 1
\end{align*}
forms a non-symmetric operad; the partial compositions are described in \cite{bala}.  (This is not a symmetric operad  even if $\mathcal{P}$ is symmetric.) Moreover, a $\mathcal{P}$-algebra structure on $A$ is given by a Maurer-Cartan element on the graded Lie algebra induced from the operad $\mathcal{O}$, or equivalently, given by a multiplication $\pi_A \in \mathcal{O}(2)$. If the operad $\mathcal{P}$ and its dual operad $\mathcal{P}^!$ are not explicitly known, one cannot construct the operad $\mathcal{O}$. However, for various Loday-type algebras, we construct the operad $\mathcal{O}$ from a different intuition, motivated from the case of dialgebras and dendriform algebras. More precisely, for various Loday-type algebras, there is a sequence of non-empty sets $\mathcal{U} = \{ U_n |~ n \geq 1 \}$ such that the collection of spaces $\{ \mathcal{O} (n) = \text{Hom}_{\mathbb{K}} (\mathbb{K} [U_n] \otimes A^{\otimes n} , A) \}_{n \geq 1}$ forms a non-symmetric operad. We observed that there is a collection of `nice' functions
\begin{align*}
R_0 (m ; \underbrace{1, \ldots,1, \overbrace{n}^{i\text{-th}} ,1, \ldots, 1 }_m) :~& U_{m+n-1} \rightarrow U_m, \\
R_i (m ; \underbrace{1, \ldots,1, \overbrace{n}^{i\text{-th}} ,1, \ldots, 1}_m ) :~& U_{m+n-1} \rightarrow \mathbb{K}[U_n], \quad 1 \leq i \leq m
\end{align*}
such that the partial compositions of the operad $\mathcal{O}$ are given by
\begin{align}\label{loday-par}
&(f \circ_i g)(r; a_1, \ldots, a_{m+n-1}) \\
&=  f (R_0 (m ; 1, \ldots, n , \ldots, 1 )r;~ a_1, \ldots, a_{i-1}, g (R_i (m ; 1, \ldots, n , \ldots, 1 ) r;~ a_i, \ldots, a_{i+n-1}), \ldots, a_{m+n-1}), \nonumber
\end{align}
for $f \in \mathcal{O}(m)$, $g \in \mathcal{O}(n)$, $r \in U_{m+n-1}$ and $a_1, \ldots, a_{m+n-1} \in A$. The identity element $\text{id} \in \mathcal{O}(1)$ is given by $\text{id} (r; a) = a,$ for all $r \in U_1$ and $a \in A$. The collection of functions $\{R_0 , R_i \}$ are called the structure functions for the operad. 
The above relation between the operad $\mathcal{P}$ and the operad $\mathcal{O}$ suggests that $\text{dim } \mathcal{P}^! (n) = \sharp (U_n)$.
Finally, a $\mathcal{P}$-algebra structure on $A$ is equivalent to a multiplication on the operad $\mathcal{O}.$

A deformation of $A$ as a $\mathcal{P}$-algebra in the sense of \cite{bala} is a deformation of the corresponding Maurer-Cartan element in the graded Lie algebra obtained from the operad $\mathcal{O}$. Note that the same Maurer-Cartan element is a multiplication in the operad $\mathcal{O}$. A deformation of $A$ as a Loday-type algebra is, by definition, a deformation of the multiplication.  Hence, these two approaches are eventually the same. In other words, they give rise to the same deformation theory for Loday-type algebras.

Note that the above description of Loday-type algebras suggests writing the deformation as follows.
Let $A$ be a fixed algebra of some Loday-type with the associated multiplication given by $\pi \in \mathcal{O}(2) = \text{Hom}_{\mathbb{K}} (\mathbb{K}[U_2] \otimes A^{\otimes 2}, A).$ Thus, a deformation of $A$ 
is given by a formal sum $\pi_t = \pi + \pi_1 t + \pi_2 t^2 + \cdots \in \mathcal{O}(2)[[t]]$ satisfying $\pi_t \circ \pi_t = 0$. Cohomological interpretations of deformation will be remarked in the next section when we introduce the cohomology of Loday-type algebras (Remark \ref{def-rem}).

\subsection{Dialgebras}
The operad for dialgebras was constructed in \cite{maj-mukh2} and the deformation theory was studied in \cite{maj-mukh}.
\begin{defn} \cite{loday}
A dialgebra is a vector space $A$ together with two bilinear maps   $ \dashv, \vdash : A \otimes A \rightarrow A$ satisfying the following relations
\begin{align*}
& a \dashv (b \dashv c) = (a \dashv b) \dashv c = a \dashv (b \vdash c),\\
& (a \vdash b) \dashv c = a \vdash (b \dashv c),\\
& (a \dashv b) \vdash c = a \vdash (b \vdash c) = (a \vdash b ) \vdash c, ~~~ \text{ for all } a, b , c \in A.
\end{align*}
\end{defn}

Let $Y_n$ be the set of planar binary trees with $(n+1)$ leaves, one root and each vertex is trivalent. The cardinality of $Y_n$ is given by the $n$-th Catalan number. For each $y \in Y_n$, we label the $(n+1)$ leaves by $\{0, 1, \ldots, n \}$ from left to right. We define maps $d_i : Y_n \rightarrow Y_{n-1}$ ($0 \leq i \leq n$) which is obtained by deleting the $i$-th leaf. Finally, the structure maps are given by
\begin{align*}
R_0 (m; 1, \ldots, n, \ldots, 1) :=~& \widehat{d_0} \circ \widehat{d_1} \circ \cdots \circ \widehat{d_{i-1}} \circ d_i \circ \cdots  \circ d_{i+n-2} \circ \widehat{d_{i+n-1}} \circ \cdots  \circ \widehat{d_{m+n-1}} : Y_{m+n-1} \rightarrow Y_m,\\
R_i (m; 1, \ldots, n, \ldots, 1) :=~& d_0 \circ d_1 \circ \cdots \circ {d_{i-2}} \circ \widehat{d_{i-1}} \circ \cdots  \circ \widehat{d_{i+n-1}} \circ d_{i+n} \circ \cdots  \circ d_{m+n-1}  : Y_{m+n-1} \rightarrow \mathbb{K}[Y_n].
\end{align*}
In such a case, the image of $R_i (m; 1, \ldots, n, \ldots, 1)$ lies in $Y_n$.
It is shown in \cite{maj-mukh2} that 
for any vector space $A$, the spaces $\mathcal{O}(n) := \text{Hom}_{\mathbb{K}} (\mathbb{K}[Y_n] \otimes A^{\otimes n}, A)$, for $n \geq 1$, defines an operad  whose partial compositions are given by (\ref{loday-par}). Moreover, if $(A, \dashv, \vdash)$ is a dialgebra, then the element $\pi \in \mathcal{O} (2)$ given by 
$\pi ( \begin{tikzpicture}[scale=0.08] \draw (0,0) -- (0, -2); \draw (2,2) -- (0,0) -- (-2,2);  \draw (1,1) -- (0,2); \end{tikzpicture}; a,b) = a \dashv b$ and
 $\pi (\begin{tikzpicture}[scale=0.08] \draw (0,0) -- (0, -2); \draw (2,2) -- (0,0) -- (-2,2); \draw (-1,1) -- (0,2); \end{tikzpicture} ; a, b) = a \vdash b$ defines a multiplication on $\mathcal{O}$.

A deformation of a dialgebra $(A, \dashv, \vdash)$ in the sense of \cite{maj-mukh} is given by two formal power series
\begin{align*}
 \dashv_t =~ \dashv_0 + \dashv_1 t + \dashv_2 t^2 + \cdots ~~~\text{ and }~~~ \vdash_t =~ \vdash_0 + \vdash_1 t + \vdash_2 t^2 + \cdots
 \end{align*}
(with $\dashv_0 =~ \dashv$ and $\vdash_0 =~ \vdash$) of binary operations on $A$ such that $(A[[t]], \dashv_t , \vdash_t)$ forms a dialgebra over $\mathbb{K}[[t]]$. Then it follows that the formal sum $\pi_t = \pi_0 + \pi_1 t + \pi_2 t^2 + \cdots \in \mathcal{O}(2)[[t]]$ defines a deformation of the multiplication $\pi$, where $\pi_i ( \begin{tikzpicture}[scale=0.08] \draw (0,0) -- (0, -2); \draw (2,2) -- (0,0) -- (-2,2);  \draw (1,1) -- (0,2); \end{tikzpicture} ; a, b ) = a \dashv_i b$ and $\pi_i ( \begin{tikzpicture}[scale=0.08] \draw (0,0) -- (0, -2); \draw (2,2) -- (0,0) -- (-2,2); \draw (-1,1) -- (0,2); \end{tikzpicture} ; a, b) = a \vdash_i b$, for all $i \geq 0$. This follows as the deformation equations of \cite{maj-mukh} are equivalent to the deformation equations (\ref{deformation-eqn}).

\subsection{Associative trialgebras}
This type of algebra is formed by three binary operations $\dashv, \vdash, \perp$ which satisfy $11$ associative-style identities. See \cite{loday-ronco} for the definition. They are related to planar trees (not necessarily binary) exactly in the same way dialgebras are related to planar binary trees.

Let $T_n$ be the set of planar trees with $n+1$ leaves and one root. Then $T_2$ has $3$ elements and $T_3$ has $11$ elements \cite{loday-ronco}. Exactly, in the same way as above, the sets $\mathcal{O}(n) := \mathrm{Hom}_{\mathbb{K}} (\mathbb{K}[T_n] \otimes A^{\otimes n}, A)$, for $n \geq 1$, inherits a structure of an operad. Further, if $(A, \dashv, \vdash, \perp)$ is an associative trialgebra, then the element $\pi \in \mathcal{O}(2)$ given by $\pi (\begin{tikzpicture}[scale=0.08] \draw (0,0) -- (0, -2); \draw (2,2) -- (0,0) -- (-2,2); \draw (1,1) -- (0,2); \end{tikzpicture}; a, b) = a \dashv b$, $\pi (\begin{tikzpicture}[scale=0.08] \draw (0,0) -- (0, -2); \draw (2,2) -- (0,0) -- (-2,2); \draw (-1,1) -- (0,2); \end{tikzpicture} ; a, b) = a \vdash b$ and $\pi (\begin{tikzpicture}[scale=0.08] \draw (0,0) -- (0, -2); \draw (2,2) -- (0,0) -- (-2,2); \draw (0,0) -- (0,2); \end{tikzpicture}; a, b) = a \perp b$ defines a multiplication on the operad $\mathcal{O}$ \cite{yau}. Note that the condition $(\pi \circ \pi) (y; a, b, c) = 0$, for all $y \in T_3$ corresponds to $11$ defining identities of associative trialgebra.

Thus, a deformation of $(A, \dashv, \vdash, \perp)$ is given by a formal sum $\pi_t = \pi + \pi_1 t + \pi_2 t^2 + \cdots \in \mathcal{O}(2)[[t]]$ satisfying $\pi_t \circ \pi_t = 0$. Explicitly, it is given by three formal power series 
\begin{center}
$\dashv_t = ~\dashv_0 + \dashv_1 t +  \dashv_2 t^2 + \cdots$,  ~~~~ $\qquad$~~~~ $\vdash_t = ~\vdash_0 + \vdash_1 t + \vdash_2 t^2 + \cdots$, \\
$\perp_t = ~\perp_0 + \perp_1 t + \perp_2 t^2 + \cdots$
\end{center}
(with $\dashv_0 =~ \dashv$, $\vdash_0 =~ \vdash$ and $\perp_0 =~ \perp$) of binary operations on $A$ such that $(A[[t]],~ \dashv_t,~ \vdash_t,~ \perp_t)$ is a associative trialgebra over $\mathbb{K}[[t]]$. These two interpretations are related by $\pi_i (\begin{tikzpicture}[scale=0.08] \draw (0,0) -- (0, -2); \draw (2,2) -- (0,0) -- (-2,2); \draw (1,1) -- (0,2); \end{tikzpicture}; a, b) = a \dashv_i b$, $\pi_i (\begin{tikzpicture}[scale=0.08] \draw (0,0) -- (0, -2); \draw (2,2) -- (0,0) -- (-2,2); \draw (-1,1) -- (0,2); \end{tikzpicture} ; a, b) = a \vdash b$ and $\pi_i (\begin{tikzpicture}[scale=0.08] \draw (0,0) -- (0, -2); \draw (2,2) -- (0,0) -- (-2,2); \draw (0,0) -- (0,2); \end{tikzpicture} ; a, b) = a \perp_i b$, for all $i \geq 0$.

\subsection{Dendriform algebras}
Dendriform algebras are Koszul dual to dialgebras and they can be thought of as a certain splitting of associative algebras \cite{loday}. The corresponding operad for dendriform algebras and deformation theory was explicitly studied in \cite{das4}.
\begin{defn}
A dendriform algebra is a vector space $A$ together with two bilinear maps $\prec,~ \succ :  A \otimes A \rightarrow A$ satisfying
\begin{align*}
& (a \prec b) \prec c = a \prec (b \prec c + b \succ c),\\
& (a \succ b) \prec c =  a \succ (b \prec c),\\
& (a \prec b + a \succ b) \succ c = a \succ (b \succ c), ~~~ \text{ for all } a, b , c \in A.
\end{align*}
\end{defn}

It follows from the definition that the sum operation $a \star b = a \prec b + a \succ b$ is associative. To define the operad,
let $C_n$ be the set of first $n$ natural numbers. Since we will treat them as symbols, we denote the elements of $C_n$ by $ \{ [1], \ldots, [n] \}.$ For any $m, n \geq 1$ and $1 \leq i \leq  m,$ we define maps $R_0 (m; 1, \ldots, n, \ldots, 1) : C_{m+n-1}  \rightarrow C_m$ and $R_i (m; 1, \ldots, n, \ldots, 1) : C_{m+n-1}  \rightarrow \mathbb{K}[C_n]$ by
\begin{align*} R_0 (m; 1, \ldots, n, \ldots, 1) ([r]) ~=~&
\begin{cases} [r] ~~~ &\text{ if } ~~ r \leq i-1 \\ [i] ~~~ &\text{ if } i \leq r \leq i +n -1 \\
[r -n + 1] ~~~ &\text{ if } i +n \leq r \leq m+n -1, \end{cases}\\
 R_i (m; 1, \ldots,  n, \ldots, 1) ([r]) ~=~&
\begin{cases} [1] + [2] + \cdots + [n] ~~~ &\text{ if } ~~ r \leq i-1 \\ [r - i+ 1] ~~~ &\text{ if } i \leq r \leq i +n -1 \\
[1]+ [2] + \cdots + [n] ~~~ &\text{ if } i +n \leq r \leq m+n -1. \end{cases}
\end{align*}

We may view these functions in the following combinatorial way. Put the first $(m+n-1)$ natural numbers into $m$ boxes in the following way
\begin{align}
\framebox{ 1} \quad \framebox{ 2} \quad \cdots \quad \framebox{ i-1} \quad \framebox{i \quad i+1 \quad  ... \quad i+n-1} \quad \framebox{i+n} \quad \cdots \quad \framebox{m+n-1}~.
\end{align}

\medskip

\noindent With these notations, the map $R_0 (m; 1, \ldots, n, \ldots, 1) ([r])$ gives the number of the box where $r$ appears and $R_i (m; 1, \ldots, n, \ldots, 1) ([r])$ gives the position of $r$ in the $i$-th box (if $r$ lies in the $i$-th box) and $[1]+ \cdots + [n]$ otherwise.

Let $A$ be a vector space. For any $n \geq 1$, we define $\mathcal{O}(n) := \mathrm{Hom}_{\mathbb{K}} (\mathbb{K}[C_n] \otimes A^{\otimes n}, A)$. Then it is shown in \cite{das4} that $\mathcal{O}$ inherits a structure of an operad with structure functions are given by (\ref{loday-par}).

Note that if $(A, \prec, \succ)$ is a dendriform algebra, then the element $\pi \in \mathcal{O}(2) = \text{Hom}_\mathbb{K}(\mathbb{K}[C_2] \otimes A^{\otimes 2}, A)$ given by $\pi ([1]; a, b) =  a \prec b$ and $\pi ([2]; a, b) = a \succ b$, defines a multiplication in the operad. It has been implicitly shown in \cite{das4} that a formal deformation of $(A, \prec, \succ)$ is equivalent to a deformation of the multiplication $\pi$.

\subsection{Dendriform trialgebras}
These algebras are Koszul dual to associative trialgebras and can be thought of as a splitting of an associative algebra by three operations \cite{loday-ronco}.
\begin{defn}
A dendriform trialgebra is a vector space $A$ endowed with three binary operations $\prec (\text{left}), ~ \succ (\text{right})$, and $\cdot ~(\text{middle})$ satisfying the following set of 7 identities
\begin{align*}
& (a \prec b) \prec c =  a \prec (b \prec c + b \succ c + b \cdot c),\\
& (a \succ b) \prec c = a \succ (b \prec c),\\
& (a \prec b + a \succ b + a \cdot b) \succ c = a \succ (b \succ c),\\
& (a \succ b) \cdot c =  a \succ (b \cdot c),\\ 
& (a \prec b) \cdot c = a \cdot (b \succ c),\\
& (a \cdot b) \prec c = a \cdot (b \prec c),\\
& (a \cdot b) \cdot c =  a \cdot (b \cdot c),~~~~ \text{ for all } a, b, c \in A.
\end{align*} 
\end{defn}
It turns out that $(A, \prec + ~\cdot,~ \succ)$ is a dendriform algebra and hence $(A, \prec + \succ +~ \cdot)$ is an associative algebra. 
To define the corresponding operad, let $P_n$ be the set of all non-empty subsets of $\{1, 2, \ldots, n\}$. Thus $P_2 = \{ \{1\}, \{2\}, \{1,2\} \} $ and $P_3 = \{ \{1\}, \{2\}, \{3\}, \{2,3\}, \{1,3\}, \{1,2\}, \{1,2,3\} \}.$

For any $m, n \geq 1$ and $1 \leq i \leq m$, we define the structure functions as
\begin{align*} R_0 (m; 1, \ldots, n, \ldots, 1) (X) ~=~& \{ r |~ \exists ~x \in X \text{ which is in } r\text{-th box} \},\\
 R_i (m; 1, \ldots, n, \ldots, 1) (X) ~=~&
\begin{cases} \sum_{S \in P_n} S ~ &\text{ if } X \cap (i\text{-th box}) = \phi \\ \{k_1 - (i-1), \ldots, k_l - (i-1) \} ~ &\text{ if } X \cap (i\text{-th box}) = \{k_1, \ldots, k_l\}, \end{cases}
\end{align*}
for any $X \in P_{m+n-1}.$
For any vector space $A$, we define $\mathcal{O}(n) = \mathrm{Hom}_\mathbb{K} (\mathbb{K} [P_n] \otimes A^{\otimes n}, A)$, for $n \geq 1$. It can be verified that $\mathcal{O}$ is an operad with the structure functions as defined above and the partial compositions are given by (\ref{loday-par}).

If $(A, \prec, \succ, \cdot)$ is a dendriform trialgebra, we define an element $\pi \in \mathcal{O}(2) = \text{Hom}_\mathbb{K} (\mathbb{K}[P_2] \otimes A^{\otimes 2}, A)$ by
$\pi (\{1\}; a, b) = a \prec b,~
\pi (\{2\}; a, b) = a \succ b,$ and
$\pi (\{1,2\}; a, b) = a \cdot b.$ 
The 7 defining identities of a dendriform trialgebra is equivalent to the fact that $\pi$ defines a multiplication on $\mathcal{O}.$
Note that the element $\pi$ can be understood by the following Hasse diagram of the set of all non-empty subsets of $\{1,2\}$ ordered by inclusion:
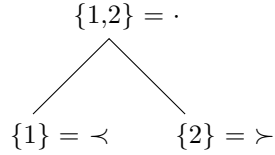
\begin{center}
\begin{tikzpicture}[scale=0.50]
\draw (-2,-2) node[anchor=north] { \qquad \{1\} =~ $\prec$} -- (0,0) node[anchor=south] {~ \quad \{1,2\} =~ $\cdot$} -- (2,-2) node[anchor=north] { \quad \qquad \{2\} =~ $\succ$};
\end{tikzpicture}
\end{center}

In view of previous discussions, a deformation of a dendriform trialgebra $(A, \prec, \succ, \cdot)$ is given by three formal power series 
\begin{align*}
\prec_t =~ \prec_0 + \prec_1 t + \prec_2 t^2 + \cdots, \quad \succ_t =~ \succ_0 + \succ_1 t + \succ_2 t^2 + \cdots ~~ \text{ and } \quad \cdot_t =~ \cdot_0 +~ \cdot_1 t +~ \cdot_2 t^2 + \cdots
\end{align*} 
(with $\prec_0 = \prec~, \succ_0 = \succ$ and $\cdot_0 = \cdot$) of binary operations on $A$ such that $(A[[t]], \prec_t,~ \succ_t,~ \cdot_t)$ is a dendriform trialgebra over $\mathbb{K}[[t]]$.

\begin{remark}
Let $(A, \prec, \succ, \cdot)$ be a dendriform trialgebra and $(\prec_t,~ \succ_t,~ \cdot_t)$ be a deformation of it. 
Then $(A[[t]], \prec_t + ~\cdot_t,~ \succ_t)$ is a dendriform algebra over $\mathbb{K}[[t]]$ which provides a deformation of the corresponding dendriform algebra $(A, \prec +~ \cdot,~ \succ)$. Moreover, the pair $(A[[t]], \prec_t + \succ_t +~ \cdot_t)$ is a deformation of the corresponding associative algebra $(A, \prec + \succ + ~\cdot).$
\end{remark}

\medskip

A Rota-Baxter algebra is an associative algebra $(A, \mu)$ together with a $\mathbb{K}$-linear map $R : A \rightarrow A$ which satisfies
\begin{align*}
\mu (R(x), R(y)) = R \big(  \mu (x, R(y)) + \mu (R(x), y) + \lambda \mu (x, y) \big), \text{ for all } x, y \in A.
\end{align*}
Here $\lambda \in \mathbb{K}$ is fixed and is called the weight of the Rota-Baxter algebra.
It follows from \cite{ebra} that a Rota-Baxter algebra $(A, \mu, R)$ of weight $\lambda$ induces a dendriform trialgebra $(A, \prec, \succ, \cdot)$ where
\begin{align*}
x \prec y = \mu (x, R(y)), ~~~\qquad  ~~~ x \succ y = \mu (R(x), y) ~~~ \qquad ~~~ \text{ and } ~~~ \qquad ~~~ x \cdot y = \lambda ~\mu (x, y).
\end{align*}
Therefore, one gets a dendriform algebra $(A, \prec', \succ')$ with
\begin{align*}
x \prec' y = x \prec y + x \cdot y  ~~~ \text{ and } ~~~ x \succ' y =  x \succ y.
\end{align*}

A deformation of a Rota-Baxter algebra $(A, \mu, R)$ of weight $\lambda$ is given by a deformation $\mu_t = \sum_{i \geq 0} \mu_i t^i$ of the associative algebra $(A, \mu)$ and a formal sum $R_t = \sum_{i \geq 0} R_i t^i$ where each $R_i \in \text{End} (A, A)$ with $R_0 = R$ such that $(A[[t]], \mu_t, R_t)$ is a weight $\lambda$ Rota-Baxter algebra over $\mathbb{K}[[t]]$.
Thus, a deformation of a Rota-Baxter algebra induces a dendriform trialgebra $(A[[t]], \prec_{A[[t]]}, \succ_{A[[t]]}, \cdot_{A[[t]]})$ and a dendriform algebra $(A[[t]], \prec'_{A[[t]]}, \succ'_{A[[t]]})$ over $\mathbb{K}[[t]]$. In other words, the triplet $(\prec_t,~ \succ_t,~ \cdot_t)$ is a deformation of the corresponding dendriform trialgebra $(A, \prec, \succ, \cdot)$ where
\begin{align*}
x \prec_t y = \sum_{n} \sum_{i+j=n} \mu_i (x, R_j (y)) t^n,  \quad
x \succ_t y = \sum_{n} \sum_{i+j=n} \mu_i (R_j (x), y) t^n ~~ \text{ and } ~~
x \cdot_t y = \sum_{n} \lambda ~\mu_n (x, y) t^n.
\end{align*}
Thus, the pair $(\prec_t', \succ_t')$ is a deformation of the corresponding dendriform algebra $(A, \prec', \succ')$ where
\begin{align*}
x \prec'_t y :=~& x \prec_t y  + x \cdot_t y = \sum_{n} \sum_{i+j=n} \mu_i (x, R_j (y)) t^n + \sum_n \lambda ~ \mu_n (x, y) t^n, \\
x \succ_t' y :=~& x \succ_t y = \sum_{n} \sum_{i+j=n} \mu_i (R_j (x), y) t^n.
\end{align*}

\subsection{Quadri-algebras}\label{subsec-quadri}
Quadri-algebras are splittings of dendriform algebras that arise naturally on the space of linear endomorphisms of an infinitesimal bialgebra and from two commuting Rota-Baxter operators \cite{aguiar-loday}. These algebras are given by $4$ binary operations  
 $\nwarrow \text{(north-west)},~ \nearrow \text{(north-east)},~ \swarrow \text{(south-west)},~ \searrow \text{(south-east)}$ and satisfying $9$ identities. See \cite{aguiar-loday} for the definition. 
We show that a quadri-algebra structure on a vector space can be seen as multiplication in a certain operad in which the structure functions are the cartesian product of the structure functions of the operad defined for dendriform algebras.

Let $Q_n = \{1, \ldots, n \} \times \{ 1, \ldots, n \}$ be the cartesian product of first $n$ natural numbers with itself. We write the elements of $Q_n$ as $\{ (1,1), \ldots, (1,n), \ldots, (n,1), \ldots, (n,n)\}.$  Thus the cardinality of $Q_n$ is $n^2$. It is useful to think that the $n^2$ elements of $Q_n$ have been allotted in a $n \times n$ square matrix.

\medskip

For any $m, n \geq 1$ and  $1 \leq i \leq m$, we define the structure maps as

$R_0 (m; 1, \ldots,  n, \ldots, 1) (r,s) ~=~$
\begin{align*} 
 \begin{cases} (r,s) ~~~ &\text{ if } ~~ 1 \leq r \leq i-1 \text{ and } 1 \leq s \leq i-1 \\ 
(i,s) &\text{ if } ~~ i \leq r \leq i+n-1 \text{ and } 1 \leq s \leq i-1 \\
(r-n+1, s) &\text{ if } ~~ i+n \leq r \leq m+n-1 \text{ and } 1 \leq s \leq i-1 \\
(r, i) &\text{ if } ~~ 1 \leq r \leq i-1  \text{ and } i \leq s \leq  i+n-1 \\
(i,i) &\text{ if } ~~ i \leq r \leq i+n-1 \text{ and } i \leq s \leq  i+n-1 \\
(r -n+1 , i) &\text{ if } ~~ i+n \leq r \leq m+n-1 \text{ and }  i \leq s \leq  i+n-1 \\
(r, s-n+1) &\text{ if } ~~ 1 \leq r \leq i-1 \text{ and } i+n \leq s \leq m+n-1\\
(i, s - n+1) &\text{ if } ~~ i \leq r \leq i+n-1 \text{ and } i+n \leq s \leq m+n-1 \\
(r - n+1 , s- n+1) &\text{ if } ~~ i+n \leq r \leq m+n-1 \text{ and } i+n \leq s \leq m+n-1, \end{cases}
\end{align*} 
and 

$R_i (m; 1, \ldots,  n, \ldots, 1) (r,s) ~=~$
\begin{align*} 
 \begin{cases}  
\sum _{(u,v) \in Q_n} (u,v)                           &\text{ if } ~~ 1 \leq r \leq i-1 \text{ and } 1 \leq s \leq i-1 \\ 
(r - i +1, 1)+ \cdots + (r-i+1, n) &\text{ if } ~~ i \leq r \leq i+n-1 \text{ and } 1 \leq s \leq i-1 \\
\sum _{(u,v) \in Q_n} (u,v)                          &\text{ if } ~~ i+n \leq r \leq m+n-1 \text{ and } 1 \leq s \leq i-1 \\
(1, s-i+1) + \cdots + (n, s-i+1)   &\text{ if } ~~ 1 \leq r \leq i-1  \text{ and } i \leq s \leq  i+n-1 \\
(r-i+1, s-i+1)                     &\text{ if } ~~ i \leq r \leq i+n-1 \text{ and } i \leq s \leq  i+n-1 \\
(1, s-i+1) + \cdots + (n, s-i+1)   &\text{ if } ~~ i+n \leq r \leq m+n-1 \text{ and }  i \leq s \leq  i+n-1 \\
\sum _{(u,v) \in Q_n} (u,v)                            &\text{ if } ~~ 1 \leq r \leq i-1 \text{ and } i+n \leq s \leq m+n-1\\
(r-i+1, 1) + \cdots + (r-i+1, n )  &\text{ if } ~~ i \leq r \leq i+n-1 \text{ and } i+n \leq s \leq m+n-1 \\
\sum _{(u,v) \in Q_n} (u,v)                           &\text{ if } ~~ i+n \leq r \leq m+n-1 \text{ and } i+n \leq s \leq m+n-1, \end{cases}
\end{align*} 
for $(r,s) \in Q_{m+n-1}$.
One observe that these functions are cartesian products of the respective functions defined for dendriform algebras. It follows that, for any vector space $A$, the spaces  $\mathcal{O}(n) = \text{Hom}_\mathbb{K} (\mathbb{K}[Q_n] \otimes A^{\otimes n}, A)$, for $n \geq 1$, inherits a structure of an operad whose structure functions are defined above and partial compositions are given by (\ref{loday-par}).

If $(A, ~\nwarrow,~ \nearrow,~ \swarrow,~ \searrow)$ is a quadri-algebra, then it can be shown that the element $\pi \in \mathcal{O}(2) = \text{Hom}_\mathbb{K}(\mathbb{K}[Q_2] \otimes A^{\otimes 2}, A)$ defined by
$\pi  ((1,1); a, b) = a \nwarrow b,~
\pi  ((1,2); a, b) = a \nearrow b,~
\pi  ((2,1); a, b) = a \swarrow b$ and
$\pi  ((2,2); a, b) = a \searrow b$
is a multiplication on the operad $\mathcal{O}.$ The correspondence between $\pi$ and $4$ operations of the quadri-algebra $A$ can be understood by the following diagram

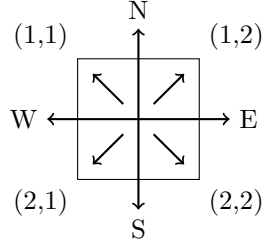
\begin{center}
\begin{tikzpicture}[scale=0.4]
\draw  (0,0) node[anchor=north east] {(2,1)} -- (4,0) node[anchor=north west] {(2,2)} -- (4,4) node[anchor=south west] {(1,2)} -- (0,4) node[anchor=south east] {(1,1)} -- (0,0); \draw[thick,->] (2,2) -- (2,5) node[anchor=south] {N} ; \draw[thick,->] (2,2) -- (5,2) node[anchor=west] {E}; \draw[thick,->] (2,2) -- (-1,2) node[anchor=east] {W}; \draw[thick,->] (2,2) -- (2,-1) node[anchor=north] {S}; \draw[thick,->] (2.5, 2.5) -- (3.5, 3.5); \draw[thick,->] (2.5, 1.5) -- (3.5, 0.5) ; \draw[thick,->] (1.5, 2.5) -- (0.5, 3.5);  \draw[thick,->] (1.5, 1.5) -- (0.5, 0.5);
\end{tikzpicture}
\end{center}

It follows from the early discussion of Section \ref{sec-4} that if $\mathcal{Q}$ is the operad for quadri-algebras, then we will have  $\text{dim } \mathcal{Q}^! (n) = n^2$. This is in favour of a conjecture made by Aguiar and Loday \cite{aguiar-loday}. 

A deformation of a quadri-algebra $(A,~ \nwarrow,~ \nearrow ,~ \swarrow,~ \searrow)$ is a deformation of $\pi$ in the operad $\mathcal{O}$. One can also explicitly write the deformation of $A$ by $4$ formal power series $(\nwarrow_t,~ \nearrow_t ,~ \swarrow_t,~ \searrow_t)$ of binary operations on $A$ such that $(A[[t]],~ \nwarrow_t,~ \nearrow_t ,~ \swarrow_t,~ \searrow_t )$ is a quadri-algebra over $\mathbb{K}[[t]]$.

\subsection{Ennea-algebras}\label{subsec-ennea}
Like quadri-algebras are splittings of dendriform algebras, ennea-algebras are splittings of dendriform trialgebras \cite{leroux}. These algebras are given by $9$ binary operations and satisfying $49$ relations. See the above reference for the definition.

We have seen that the structure functions for the operad of quadri-algebras are cartesian products of the structure functions for the operad of dendriform algebras. Similarly, the structure functions for the operad of ennea-algebras are cartesian products of the structure functions for the operad of dendriform trialgebras. More precisely, 
define $E_n = P_n \times P_n$, where $P_n$ is the set of all non-empty subsets of $\{1, \ldots, n \}$. It follows that the cardinality of $E_2$ is $9$ and that of $E_3$ is $49$. We define the structure functions  $R_0 (m; 1, \ldots, 1, n, 1, \ldots, 1) : E_{m+n-1} \rightarrow E_m$ and 
$R_i (m; 1, \ldots, 1, n, 1, \ldots, 1) : E_{m+n-1} \rightarrow \mathbb{K}[E_n]$ to be the cartesian product of the structure functions defined for dendriform trialgebras.
The $9$ elements of $E_2$ correspond to $9$ binary operations and $49$ elements of $E_3$ correspond to $49$ defining relations of an ennea-algebra.
More explicitly, an ennea-algebra structure on a vector space $A$ is equivalent to a multiplication on the operad $\mathcal{O}(n) = \text{Hom}_\mathbb{K}(\mathbb{K}[E_n] \otimes A^{\otimes n}, A)$, for $n \geq 1$, whose structure functions are mentioned above. Deformations of ennea-algebras can be defined in an analogous way.

It follows that if $\mathcal{E}$ is the operad for ennea-algebras, then we will have $\text { dim } \mathcal{E}^! (n) = (2^n -1)^2$.

\subsection{Hom analog of Loday-type algebras}
Recently hom-type algebras have been studied by many authors. In these algebras, the identities defining the structures are twisted by one homomorphism (or two commuting homomorphisms). See \cite{makh-sil, amm-ej-makh, gra-makh-men-pana} for more details. In this subsection, we describe Loday-type algebras twisted by homomorphisms as multiplication in certain twisted operads.

Let $(\mathcal{O}, \gamma, \text{id})$ be an operad with partial compositions $\circ_i$. Let $\alpha, \beta \in \mathcal{O}(1)$ be such that $\alpha \circ \beta = \beta \circ \alpha.$ Consider
\begin{align*}
\mathcal{O}_{\alpha, \beta} (n) = \{ f \in \mathcal{O}(n) |~ \gamma (f ; \alpha, \ldots, \alpha) = \alpha \circ f ~\text{and}~  \gamma (f ; \beta, \ldots, \beta) = \beta \circ f \},~~ \text{ for } n \geq 1.
\end{align*}
Define twisted partial compositions 
$\circ_i' : \mathcal{O}_{\alpha, \beta} (m) \otimes \mathcal{O}_{\alpha, \beta} (n) \rightarrow \mathcal{O}_{\alpha, \beta} (m+n-1)$ by
\begin{align*}
(f \circ_i' g ) = ~ \gamma (f ;~ \alpha^{n-1}, \ldots, \alpha^{n-1}, \underbrace{g}_{i\text{-th place}}, \beta^{n-1}, \ldots, \beta^{n-1}),
\end{align*}
for $f \in \mathcal{O}_{\alpha, \beta}(m),~g \in \mathcal{O}_{\alpha, \beta}(n)$ and $1 \leq i \leq m$.

\begin{prop}
With the above notations $(\mathcal{O}_{\alpha, \beta}, \circ_i', \mathrm{id})$ forms an operad.
\end{prop}

The proof of the above proposition is simple and can be found in \cite{das2, das3} when $\mathcal{O}$ is the endomorphism operad $\text{End} (A)$ associated to a vector space $A$. The operad $(\mathcal{O}_{\alpha, \beta}, \circ_i', \text{id})$ is called the twisted variation of $\mathcal{O}$ twisted by $\alpha$ and $\beta$. One observes that when $\alpha = \beta = \text{id} \in \mathcal{O}(1)$, the twisted variation $\mathcal{O}_{\text{id}, \text{id}}$ is same as the operad $\mathcal{O}.$

 Thus it follows from the previous proposition that one can construct a twisted version of various operads as defined in previous subsections. Multiplications of these twisted operads are called twisted algebras. For example, a twisted associative algebra (also called BiHom-associative algebra) structure on $A$ is given by a multiplication on the twisted operad $\text{End}_{\alpha, \beta}(A)$ \cite{das3}. When $\alpha = \beta$, one get Hom-associative algebras.
 
One can construct twisted algebra structures as follows. Suppose $\pi \in \mathcal{O}(2) = \text{Hom}_\mathbb{K}(\mathbb{K}[U_2] \otimes A^{\otimes 2}, A)$ defines a fixed Loday-type algebra structure on $A$, and $\alpha, \beta$ be two commuting algebra morphisms. That is, $\alpha \circ \pi = \{ \pi \} \{ \alpha , \alpha \}$ and $\beta \circ \pi = \{ \pi \} \{ \beta , \beta \}$. Then $\{ \pi \} \{ \alpha, \beta \} \in \mathcal{O}_{\alpha, \beta} (2)$ given by
 \begin{align*}
\{ \pi \} \{ \alpha, \beta \} (r; a, b) := \pi (r; \alpha (a), \beta (b)),
\end{align*}
for $ r \in U_2,~ a, b \in A$, defines a twisted Loday-type algebra structure on $A$. This construction is called the `Yau twist'.

In the same analogy, a deformation of a twisted Loday-type algebra $A$ is by definition a deformation of the corresponding multiplication in the twisted variation operad. For (Bi)Hom-associative algebras, one recovers the deformation described in \cite{amm-ej-makh, das3}.

\section{Cohomology of Loday-type algebras}\label{sec-5}
In this section, we study representations and cohomology of Loday-type algebras from the perspectives of multiplicative operads. We show that this cohomology is isomorphic to the operadic cohomology for Loday-type algebras.

\subsection{Representations}
Let $A$ be a fixed Loday-type algebra. Assume that the Loday-type algebra structure on $A$ can be given by a multiplication $\pi$ on the operad $\mathcal{O}$ in which the structure functions are given by $\{R_0, R_i \}$ on the sets $\{ U_n |~ n \geq 1 \}.$ See the beginning of Section \ref{sec-4}.
\begin{defn}
A representation of $A$ is given by a vector space $M$ together with $\mathbb{K}$-multilinear maps
\begin{align*}
\theta_1 : \mathbb{K}[U_2] \otimes (A \otimes M) \rightarrow M, ~~\qquad~~ \theta_2 : \mathbb{K}[U_2] \otimes (M \otimes A) \rightarrow M
\end{align*}
satisfying
\begin{align*}
\theta_1 \big(  R_0 (2; 1, 2) y; ~ a,~ \theta_1 ( R_2 (2; 1, 2)y; b, m)  \big) =~& \theta_1 (R_0 (2; 2, 1)y ; ~\pi (R_1 (2; 2, 1)y; a, b),~ m), \\
\theta_2 \big(  R_0 (2; 1, 2) y; ~ m,~ \pi ( R_2 (2; 1, 2)y; a, b)  \big) =~& \theta_2 (R_0 (2; 2, 1)y ; ~\theta_2 (R_1 (2; 2, 1)y; m, a),~ b), \\
\theta_1 \big(  R_0 (2; 1, 2) y; ~ a,~ \theta_2 ( R_2 (2; 1, 2)y; m, b)  \big) =~& \theta_2 (R_0 (2; 2, 1)y ; ~\theta_1 (R_1 (2; 2, 1)y; a, m),~ b),
\end{align*}
for all $y \in U_3$ and $a, b \in A$, $m \in M$.
\end{defn}

There are $3~ \sharp (U_3)$ relations to define a representation.
Moreover, it follows that $A$ is a representation of itself with $\theta_1 = \theta_2 = \pi$. Any vector space $M$ can be considered as  a representation of $A$ with $\theta_1 = \theta_2 = 0.$

Let $A$ be a fixed Loday-type algebra. An ideal of $A$ is a subspace $I \subset A$ satisfying $\pi (r ; A, I) \subset I$ and $\pi (r; I, A) \subset I$, for all $r \in U_2$. Any ideal of $A$ is a representation with $\theta_1 = \theta_2 = \pi$.

\begin{prop} (Semi-direct product) Let $A$ be a fixed Loday-type algebra (given by the multiplication $\pi$) and $M$ be a representation of $A$. Then the direct sum $A \oplus M$ inherits a Loday algebra structure of the same type. The multiplication is given by
\begin{align*}
\pi_{A \oplus M} (r ; (a, m), (b, n)) = (\pi (r; a, b),~ \theta_1 (r; a, n) + \theta_2 (r; m, b)),
\end{align*}
for $r \in U_2$ and $(a, m), (b, n) \in A \oplus M$.
\end{prop}

\begin{proof}
Straightforward.
\end{proof}

\subsection{Cohomology with coefficients}

Let $A$ be a Loday-type algebra and $M$ be a representation of it. Define the group of $n$-cochains by
\begin{align}\label{comp-with-co}
C^n(A, M) := \text{Hom}_\mathbb{K} (\mathbb{K}[U_n] \otimes A^{\otimes n}, M), ~~~ \text{for } n \geq 1.
\end{align}
The coboundary operator $\delta : C^n(A, M) \rightarrow C^{n+1} (A, M)$ is given by
\begin{align*}
&(\delta f) (r; a_1 , \ldots, a_{n+1}) \\
&=  \theta_1 \big( R_0 (2;1,n) r; ~a_1, f (R_2 (2;1,n)r; a_2, \ldots, a_{n+1})   \big) \\
&+ \sum_{i=1}^n  (-1)^i~ f \big(  R_0 (n; 1, \ldots, 2, \ldots, 1)r;~ a_1, \ldots, a_{i-1}, \pi (R_i (1, \ldots, 2, \ldots, 1)r; a_i, a_{i+1}), a_{i+2}, \ldots, a_{n+1}   \big) \\
&+ (-1)^{n+1} ~\theta_2 \big( R_0 (2; n, 1) r;~ f (R_1 (2;n,1)r; a_1, \ldots, a_n), a_{n+1}   \big),
\end{align*}
for $r \in U_{n+1}$ and $a_1, \ldots, a_{n+1} \in A.$ We denote the corresponding cohomology groups by $H^n(A, M)$, for $n \geq 1$.

Like classical cases, $1$-cocycles in the above cochain complex are called derivations on $A$ with values in the representation $M$.

When we consider the case of a dialgebra, our cohomology coincides with that of Frabetti \cite{frab} and in the case of a dendriform algebra, it coincides with the explicit dendriform cohomology given in \cite{das4}.

\begin{remark}\label{gers-rem}
When $M = A$ with the representation given by $\theta_1 = \theta_2 = \pi$, then up to a sign, the above coboundary map coincides with the one induced from the multiplication $\pi$ (see Section \ref{sec-2}). Therefore, it follows from the discussions of Section \ref{sec-2} that the cohomology of $A$ (with coefficients in itself) inherits a Gerstenhaber algebra structure. As a consequence, the cohomology of dialgebra, dendriform algebra, dendriform trialgebra, quadri-algebra, ennea-algebra and their hom analogs have Gerstenhaber structure on their cohomology. The first observation also ensures that $\delta^2 = 0$ for the coboundary map defined above with coefficients in any arbitrary representation.
\end{remark}
\begin{remark}\label{comp-two-cohomo}
If $\mathcal{P}$ is a non-symmetric operad defining the type of a Loday algebra $A$, then we have seen in Section \ref{sec-4} that the algebra structure on $A$ is given by a Maurer-Cartan element on the gLa induced from the operad $\mathcal{O}$. See the introduction of Section \ref{sec-4} for the operad $\mathcal{O}$. The coboundary operator for the operadic cohomology of $A$ (with coefficients in itself) is induced by the Maurer-Cartan element. Note that the same Maurer-Cartan element is a multiplication on the operad $\mathcal{O}$, and the differential induced from the multiplication is same as the one induced from the Maurer-Cartan element. Hence the corresponding cohomology groups are the same.

When one considers the cohomology of $A$ with coefficients in a representation $M$, one may first consider the semi-direct product algebra on $A \oplus M$. Then the above cochain complex (\ref{comp-with-co}) is a subcomplex of the cohomology of $A \oplus M$ with coefficients in itself. Therefore, by the same argument as above and from the definition of the operadic cohomology $H^\bullet_\mathcal{P}(A, M)$ with coefficients \cite{bala},  we have $H^\bullet_\mathcal{P}(A, M)$ coincides with the cohomology $H^\bullet (A, M)$ induced from the operad with multiplication.
\end{remark}

\begin{remark}\label{def-rem}
In the previous section, we define deformations of a Loday-type algebra $A$ as deformations of the corresponding multiplication. It follows from Remark \ref{gers-rem} that the results of Section \ref{sec-3} can be stated for Loday-type algebras as follows:
\begin{enumerate}
\item The second cohomology group of $A$ corresponds bijectively to the set of equivalence classes of infinitesimal deformations.
\item The vanishing of the second cohomology of $A$ implies that $A$ is rigid.
\item The vanishing of the third cohomology allows one to extend a finite order deformation of $A$ to the next order.
\end{enumerate}
For dialgebras and dendriform algebras, the corresponding results have been proved in \cite{maj-mukh, das4}.
\end{remark}

\subsection{Abelian extensions}
In this subsection, we show that the second cohomology group of a Loday-type algebra can be described by equivalence classes of abelian extensions. A similar result for operadic cohomology was given in \cite{bala2}. We start with the following definition.

\begin{defn}
Let $A$ and $B$ be two Loday algebras of the same type. A morphism between them is given by a linear map $f : A \rightarrow B$ satisfying
\begin{align*}
f (\pi_A (r; a, a')) = \pi_B (r; f(a), f(a')),
\end{align*}
for all $r \in U_2,~ a, a' \in A,$ where $\pi_A$ and $\pi_B$ denote the multiplications corresponding to the algebra structures on $A$ and $B$, respectively.
\end{defn}

Let $A$ be a Loday-type algebra and $M$ be a vector space. Note that $M$ can be considered as a Loday algebra of the same type with the trivial multiplication $\pi_M = 0$.

\begin{defn}
An abelian extension of $A$ by $M$ is given by an extension
\[
\xymatrix{
0 \ar[r] & M \ar[r]^{i} & E \ar[r]^{j} & A \ar[r] & 0
}
\]
of Loday algebras (of the same type) such that the sequence is split over $\mathbb{K}$.
\end{defn}

An abelian extension induces an $A$-representation on $M$ via the actions
\begin{align*}
\theta_1 (r ; a, m) =~ \pi_E (r ; s (a), i (m)) ~~~\text{ and }~~~
\theta_2 (r ; m, a) =~ \pi_E (r ; i (m), s (a)),
\end{align*}
for $r \in U_2$, $a \in A, ~m \in M$ and $s :A \rightarrow E$ is any section corresponding to the $\mathbb{K}$-splitting. One can easily verify that this action is independent of the choice of $s$.

Two such abelian extensions are said to be equivalent if there is a morphism $\phi : E \rightarrow E'$ between Loday-type algebras which makes the following diagram commute
\[
\xymatrix{
 0 \ar[r] & M \ar@{=}[d] \ar[r]^{i} & E \ar[d]^\phi \ar[r]^{j} & A \ar@{=}[d] \ar[r] & 0 \\
 0 \ar[r] & M \ar[r]_{i'} & E' \ar[r]_{j'} & A \ar[r] & 0.
}
\]

Next, fix an $A$-representation $M$. We denote by $\mathcal{E}xt (A, M)$ the equivalence classes of abelian extensions of $A$ by $M$ for which the induced representation on $M$ is the prescribed one.

\begin{thm}
There is a bijection: $H^2 (A, M) \cong \mathcal{E}xt (A, M).$
\end{thm}

\begin{proof}
Let $f \in C^2 (A, M)$ be a $2$-cocycle. We consider the $\mathbb{K}$-module $E = A \oplus M$ and define a map $\pi_E : \mathbb{K}[U_2] \otimes E^{\otimes 2} \rightarrow E$ by
\begin{align*}
\pi_E (r; (a, m), (b, n)) =~& ( \pi (r; a, b), ~\theta_1 (r; a , n) + \theta_2 ( r; m, b) + f (r; a, b)),
\end{align*}
for $r \in U_2$ and $(a,m), (b, n) \in E.$
(Observe that when $f =0$ this is the semi-direct product.)
Using the fact that $f$ is a $2$-cocycle, it is easy to verify that $\pi_E$ defines a Loday algebra structure (of the same type) on $E$. Moreover, $0 \rightarrow M \rightarrow E \rightarrow A \rightarrow 0$ defines an abelian extension with the obvious splitting. 
Let $\pi_E' : \mathbb{K}[U_2] \otimes E^{\otimes 2} \rightarrow E$ be the Loday-type algebra structure on $E$
associated to the cohomologous $2$-cocycle $f - \delta (g)$, for some $g \in C^1 (A, M)$. The equivalence between abelian extensions $(E, \pi_E) $ and $(E, \pi_E')$ is given by $(a, m) \mapsto (a, m + g (a))$. Therefore, the map $H^2 (A, M) \rightarrow \mathcal{E}xt (A, M) $ is well defined.

Conversely, given an extension 
$0 \rightarrow M \xrightarrow{i} E \xrightarrow{j} A \rightarrow 0$ with splitting $s$, we may consider $E = A \oplus M$ and $s$ is the map $s (a) = ( a, 0).$ With respect to this splitting, the maps $i$ and $j$ are the obvious ones. Since $j \circ \pi_E (r; (a, 0), (b, 0)) = \pi (r; a, b)$ as $j$ is an algebra map, we have $\pi_E (r; (a, 0), ( b, 0)) = (f (r; a, b), \pi (r; a, b))$, for some $f \in C^2 (A, M).$ 
Since $\pi_E$ defines a Loday algebra structure on $E$, it follows 
that $f$ is a $2$-cocycle. Similarly, one can observe that any two equivalent extensions are related by a map $E = A \oplus M \xrightarrow{\phi} A \oplus M = E'$, $(a, m) \mapsto (a, m + g(a))$ for some $g \in C^1 (A, M)$. Since $\phi$ is an algebra morphism, we have
\begin{align*}
\phi \circ \pi_E (r; ( a, 0), ( b, 0)) = \pi'_{E} (r; \phi ( a, 0) , \phi (b, 0))
\end{align*}
which implies that $f' (r; a, b) = f (r; a, b) - (\delta g)(r; a, b)$. Here $f'$ is the $2$-cocycle induced from the extension $E'$. This shows that the map $\mathcal{E}xt (A, M) \rightarrow H^2 (A, M)$ is well defined. Moreover, these two maps are inverses to each other. Hence the proof.
\end{proof}

\section{Deformations of morphisms}
In this section, we study deformations of morphisms between Loday-type algebras. The results of this section are parallel to the classical cases (see, for example, \cite{gers-sch, yau2}).

Let $A$ and $B$ be two Loday algebras of the same type and $f : A \rightarrow B$ be a morphism. Then $B$ can be considered as a representation of $A$ via $f$ as follows:
\begin{align*}
\theta_1 (r ; a, b) = \pi_B (r ; f(a), b) ~~~ \qquad \quad ~~~  \theta_2 (r ; b, a) = \pi_B (r; b, f(a)),
\end{align*}
for all $r \in U_2$, $a \in A$ and $b \in B$, where $\pi_B : \mathbb{K}[U_2] \otimes B^{\otimes 2} \rightarrow B$ denotes the multiplication defining the algebra structure on $B$.

In such a case, we define a new cochain complex whose $n$-th cochain group is given by 
\begin{align*}
C^n(f, f) =~& C^n (A, A) \times C^n (B, B) \times C^{n-1} (A, B) \\
=~& \text{Hom}_\mathbb{K} (\mathbb{K}[U_n] \otimes A^{\otimes n}, A)  ~\times ~\text{Hom}_\mathbb{K} (\mathbb{K}[U_n] \otimes B^{\otimes n}, B) ~ \times ~ \text{Hom}_\mathbb{K} (\mathbb{K}[U_{n-1}] \otimes A^{\otimes n-1}, B)
\end{align*}
and the differential $\delta_f : C^n (f, f) \rightarrow C^{n+1} (f, f)$ is defined by
\begin{align*}
\delta_f (\phi, \psi, \zeta ) = (\delta_A \phi,~ \delta_B \psi,~ f \circ \phi - \psi \circ f^{\otimes n} - \delta \zeta)
\end{align*}
where $\delta_A , \delta_B$ denote the coboundary maps defining the cohomology of $A$ and $B$, respectively, and $\delta$ denotes the coboundary map defining the cohomology of $A$ with coefficients in $B$.

\begin{prop}
With the above notations $(C^n (f,f), \delta_f)$ is a cochain complex.
\end{prop}

\begin{proof}
We have
\begin{align*}
(\delta_f)^2 (\phi, \psi, \zeta) =~& \delta_f (\delta_A \phi,~ \delta_B \psi,~ f \circ \phi - \psi \circ f^{\otimes n} - \delta \zeta) \\
=~& (\delta_A^2 \phi,~ \delta_B^2 \psi,~  f \circ \delta_A \phi - (\delta_B \psi) \circ f^{\otimes (n+1)} - \delta (f \circ \phi - \psi \circ f^{\otimes n} - \delta \zeta)).
\end{align*}
It follows from a direct verification that $f \circ \delta_A \phi = \delta (f \circ \phi)$ and $(\delta_B \psi) \circ f^{\otimes (n+1)} = \delta (\psi \circ f^{\otimes n})$. Hence $(\delta_f)^2 = 0$ as $\delta_A, \delta_B$ and $\delta$ are differentials.
\end{proof}

The complex $(C^\bullet (f,f), \delta_f)$ is called the deformation complex associated to the algebra morphism $f$. The corresponding cohomology groups are denoted by $H^\bullet (f,f)$. This cohomology is related to the cohomology of $A$ and $B$ in the following way.

\begin{prop}\label{when-def-mor-zero}
If $H^n (A, A), ~ H^n(B, B)$ and $H^{n-1} (A, B)$ are all trivial, then $H^n (f,f)$ is so.
\end{prop}

\begin{proof}
Let $(\phi, \psi, \zeta) \in C^n (f,f)$ be an $n$-cocycle. Then it follows that $\phi \in C^n (A, A), ~ \psi \in C^n (B, B)$ are $n$-cocycles and $f \circ \phi - \psi \circ f^{\otimes n} - \delta \zeta = 0$. Hence by the hypothesis, there exist $(n-1)$-cochains $\phi' \in C^{n-1} (A, A)$ and $\psi' \in C^{n-1}(B, B)$ such that $\phi = \delta_A \phi'$ and $\psi = \delta_B \psi'$. Moreover,
\begin{align*}
\delta (f \circ \phi' - \psi' \circ f^{\otimes (n-1)} - \zeta) =~& f \circ \delta_A \phi' - (\delta_B \psi') \circ f^{\otimes n} - \delta \zeta \\
=~& f \circ \phi - \psi \circ f^{\otimes n} - \delta \zeta = 0.
\end{align*}
Hence, $f \circ \phi' - \psi' \circ f^{\otimes (n-1)} - \zeta \in C^{n-1} (A, B)$ is an $(n-1)$-cocycle. By the hypothesis, there exists an element $\zeta' \in C^{n-2} (A, B)$ such that $f \circ \phi' - \psi' \circ f^{\otimes (n-1)} - \zeta = \delta \zeta'$. Thus, it follows that
$(\phi, \psi, \zeta) = \delta_f (\phi', \psi', \zeta')$ is a coboundary.
\end{proof}

Unlike deformations of Loday-type algebras, deformations of morphisms cannot describe by using multiplicative operads. The reason is the appearance of the third factor in the deformation complex of $f$.

\subsection{Deformations} In this subsection, we describe formal deformations of a morphism between Loday-type algebras and show that the above-defined cohomology controls such deformations.

\begin{defn}\label{defn-def-mor}
A deformation of $f$ is given by a triple $\theta_t = (\pi_{A, t}, \pi_{B, t}, f_t)$ in which
\begin{itemize}
\item $\pi_{A, t} = \sum_{i \geq 0} \pi_{A, i} t^i$ is a deformation of $A$;
\item $\pi_{B, t} = \sum_{i \geq 0} \pi_{B, i} t^i$ is a deformation of $B$;
\item $f_t = \sum_{i \geq 0} f_i t^i : A[[t]] \rightarrow B[[t]]$ is a morphism of algebras, where each $f_i : A \rightarrow B$ is a $\mathbb{K}$-linear map and $f_0 = f$.
\end{itemize}
\end{defn}

\begin{defn}
Two deformations $\theta_t = (\pi_{A, t}, \pi_{B, t}, f_t)$ and $\theta_t' = (\pi_{A, t}', \pi_{B, t}', f_t')$ of $f$ are said to be equivalent if there is a pair $\Phi_t = (\phi_{A, t}, \phi_{B, t})$ in which
\begin{itemize}
\item $\phi_{A, t} : A[[t]] \rightarrow A[[t]]$ is an equivalence between $\pi_{A, t}$ and $\pi_{A, t}'$;
\item $\phi_{B, t} : B[[t]] \rightarrow B[[t]]$ is an equivalence between $\pi_{B, t}$ and $\pi_{B, t}'$;
\item $f_t \circ \phi_{A, t} = \phi_{B, t} \circ f_t'.$
\end{itemize}
\end{defn}

\begin{prop}
The linear part $(\pi_{A, 1}, \pi_{B, 1}, f_1)$ of a deformation $\theta_t$ is a $2$-cocycle in the complex $( C^\bullet (f,f), \delta_f )$ whose cohomology class is determined by the equivalence class of $\theta_t$.
\end{prop}

\begin{proof}
Let $\theta_t = (\pi_{A, t}, \pi_{B,t}, f_t)$ be a deformation of $f$. Since $\pi_{A, t} = \sum_{i \geq 0} \pi_{A, i} t^i$ and $\pi_{B, t} = \sum_{i \geq 0} \pi_{B, i} t^i$ are deformations of $A$ and $B$, respectively, we have $\pi_{A, 1} \in \text{Hom}_\mathbb{K} (\mathbb{K}[U_2] \otimes A^{\otimes 2}, A)$ and $\pi_{B, 1} \in \text{Hom}_\mathbb{K} (\mathbb{K}[U_2] \otimes B^{\otimes 2}, B)$ are $2$-cocycles of $A$ and $B$, respectively. Moreover, $f_t : A[[t]] \rightarrow B[[t]]$ is a morphism implies that
\begin{align*}
f_t (\pi_{A, t} (r;a, b) ) = \pi_{B, t} (r, f_t(a), f_t (b)), 
\end{align*}
for all $r \in U_2$ and $a, b \in A$. By equating coefficients of $t^n$, for $n \geq 0$, we get
\begin{align*}
\sum_{i+j = n} f_i (\pi_{A, j} (r; a, b) ) = \sum_{i+j+k = n} \pi_{B, i} (r; f_j(a), f_k (b)), ~~~ \text{ for all } n \geq 0.
\end{align*}
(For $n = 0$, this identity is equivalent to the fact that $f : A \rightarrow B$ is a morphism of algebras.) For $n = 1$, we get
\begin{align*}
f (\pi_{A, 1} (r; a, b)) + f_1 (\pi_A (r;a,b)) = \pi_B (r; f(a), f_1(b)) + \pi_B (r; f_1 (a), f(b)) + \pi_{B, 1} (r; f(a), f(b)).
\end{align*}
This is equivalent to 
\begin{align*}
f \circ \pi_{A, 1} - \pi_{B, 1} \circ f^{\otimes 2} - \delta  (f_1) = 0.
\end{align*}
Therefore, we conclude that $\delta_f (\pi_{A,1}, \pi_{B,1}, f_1) = 0$.

\medskip

Finally, let $\theta_t = (\pi_{A, t}, \pi_{B, t}, f_t)$ and $\theta_t' = (\pi_{A, t}', \pi_{B,t}', f_t')$ be two equivalent deformations of $f$ and the equivalence is given by $\Phi_t = (\phi_{A, t}, \phi_{B,t}).$ Since $\phi_{A, t} : A[[t]] \rightarrow A[[t]]$ is an equivalence between the deformations $\pi_{A, t}$ and $\pi_{A, t}'$, we have $\pi_{A, 1}' - \pi_{A, 1} = \delta_A (\phi_{A, 1}).$ Similarly, we have $\pi_{B, 1}' - \pi_{B, 1} = \delta_B (\phi_{B, 1})$. Finally, the condition $f_t \circ \phi_{A, t} = \phi_{B, t} \circ f_t'$ implies that
\begin{align*}
f_1' - f_1 = f \circ \phi_{A, 1} - \phi_{B,1} \circ f.
\end{align*}
Thus, it follows that the difference $(\pi'_{A,1}, \pi'_{B, 1}, f'_1) - (\pi_{A, 1}, \pi_{B,1}, f_1) =  \delta_f ( \phi_{A, 1} , \phi_{B, 1}, 0).$ Hence the proof.
\end{proof}

The linear part $(\pi_{A,1}, \pi_{B, 1}, f_1)$ is called the infinitesimal of the deformation $\theta_t = (\pi_{A, t}, \pi_{B, t}, f_t)$. It follows that infinitesimals of deformations are $2$-cocycles and equivalent deformations have cohomologous infinitesimals. In general, if $(\pi_{A,1}, \pi_{B, 1}, f_1) = \cdots = (\pi_{A,n-1}, \pi_{B, n-1}, f_{n-1}) = (0,0,0)$, then $(\pi_{A,n}, \pi_{B, n}, f_n)$ is a $2$-cocycle.

\begin{defn}
A deformation $\theta_t = (\pi_{A, t}, \pi_{B,t}, f_t)$ of $f$ is called trivial if it is equivalent to the deformation $\theta_t' = (\pi_A, \pi_B, f)$.
A morphism $f : A \rightarrow B$ between same Loday-type algebras is called rigid if any deformation of $f$ is equivalent to a trivial deformation.
\end{defn}

\begin{prop}
A nontrivial deformation $\theta_t$ of $f$ is equivalent to a deformation $\theta_t'$ in which the first nonzero term  $(\pi_{A, p}', \pi_{B, p}', f_p')$ is a $2$-cocycle but not a coboundary.
\end{prop}

\begin{proof}
Let $\theta_t = (\pi_{A, t}, \pi_{B, t}, f_t)$ be a deformation of $f$ in which $$(\pi_{A, 1}, \pi_{B, 1}, f_1) = \cdots = (\pi_{A, n-1}, \pi_{B, n-1}, f_{n-1}) = (0, 0, 0)$$ and $(\pi_{A, n}, \pi_{B, n}, f_n)$ is the first nonzero term. Then $(\pi_{A, n}, \pi_{B, n}, f_n)
$ is a $2$-cocycle in $C^2(f, f)$. Assume that it is a coboundary, say $(\pi_{A, n}, \pi_{B, n}, f_n)= - \delta_f (\phi, \psi, 0).$ 
Hence $\pi_{A, n} = - \delta_A \phi$, $\pi_{B, n} = - \delta_B \psi$ and $f_n =  - (f \circ \phi - \psi \circ f)$. Setting
\begin{align*}
\phi_{A,t} = \text{id}_A + \phi~ t^n    ~~~~ \text{ and } ~~~~\phi_{B, t} =  \text{id}_B + \psi~ t^n.
\end{align*}
Define $\theta_t' = (\pi_{A, t}', \pi_{B, t}', f_t')$ where $\pi_{A, t}' =  \phi_{A, t}^{-1} \circ \{ \pi_{A, t} \} \{ \phi_{A, t}, \phi_{A,t} \}$, $\pi_{B, t}' = \phi_{B, t}^{-1} \circ \{ \pi_{B, t} \} \{ \phi_{B,t} , \phi_{B, t} \}$ and $f_t' = \phi_{B, t}^{-1} \circ f_t \circ \phi_{A, t}$. Then $\theta_t'$ is a deformation of $f$ in which $(\pi_{A, 1}', \pi_{B,1}', f_1') = \cdots = (\pi_{A, n}', \pi_{B,n}', f_n') = (0,0, 0)$. If the first non-zero term is not a coboundary then we are done. If not then we can apply the same method to obtain a required deformation.
\end{proof}

\begin{thm}
If $H^2 (f, f) = 0$ then $f$ is rigid.
\end{thm}

In the same spirit of Section \ref{sec-3}, we would like to extend a deformation of finite order to a deformation of next order. A deformation of order $n$ consists of a triple  $\theta_t = (\pi_{A, t}, \pi_{B, t}, f_t)$ of the form
\begin{align*}
\theta_t = (\pi_{A, t} = \sum_{i=0}^n \pi_{A, i} t^i,~ \pi_{B, t} = \sum_{i=0}^n \pi_{B,i} t^i,~ f_t = \sum_{i=0}^n f_i t^i)
\end{align*}
such that the conditions of Definition \ref{defn-def-mor} hold modulo $t^{n+1}$.
Suppose there is a $2$-cochain $\theta_{n+1} = (\pi_{A, n+1}, \pi_{B, n+1}, f_{n+1}) \in C^2 (f,f)$ such that
\begin{align*}
\overline{\theta}_t = \theta_t + \theta_{n+1} t^{n+1} = ( \sum_{i=0}^{n+1} \pi_{A, i} t^i,~ \sum_{i=0}^{n+1} \pi_{B,i} t^i,~ \sum_{i=0}^{n+1} f_i t^i)
\end{align*}
is a deformation of order $n+1$.
It turns out that the first two components of $\theta_{n+1} = (\pi_{A, n+1}, \pi_{B, n+1}, f_{n+1})$ must satisfy (see equation (\ref{alt-def-eqn}))
\begin{itemize}
\item $- \delta_A (\pi_{A, n+1}) = \underbrace{ \sum_{i+j = n+1, i, j \geq 1} \pi_{A, i} \circ \pi_{A, j}}_{\text{Ob}_A},$
\item $- \delta_B (\pi_{B, n+1}) = \underbrace{ \sum_{i+j = n+1, i, j \geq 1} \pi_{B, i} \circ \pi_{B, j}}_{\text{Ob}_B}.$
\end{itemize}
One may also define a map $\theta(f) : \mathbb{K}[U_2] \otimes A^{\otimes 2} \rightarrow B$ by
\begin{align*}
\theta(f) (r; a, b) :=  \sum_{i=1}^n f_i (\pi_{A, n+1-i} (r; a, b)) - \sum_{i+j+k = n+1, 0 \leq i, j, k < n+1}  \pi_{B, i} (r; f_j (a), f_k (b)),
\end{align*}
for $r \in U_2$ and $a, b \in A$. The triple
\begin{align*}
\text{Ob} (\theta_t) = (\text{Ob}_A, \text{Ob}_B, \theta(f)) \in C^3 (f, f)
\end{align*}
is called the obstruction to extend the deformation.

The proof of the following proposition is a tedious calculation and similar to the dialgebra case \cite{yau2}.
\begin{prop}
The obstruction $\mathrm{Ob} (\theta_t)$ is a $3$-cocycle.
\end{prop}

Hence we obtain the following.

\begin{thm}
If $H^3 (f, f) = 0$ then every finite order deformation of $f$ can be extended to a deformation of next order. In such a case, every $2$-cocycle in $C^2 (f, f)$ is the infinitesimal of some deformation of $f$.
\end{thm}

When we consider the case of dialgebra morphisms, our theory compactify the descriptions of \cite{yau2}. Similarly, we can apply the results of this section to morphisms between other Loday-type algebras.

\medskip

\noindent {\bf Acknowledgements.} The author would like to thank the anonymous referee for his/her valuable comments that very much improved the paper. He also wishes to thank Andrea Solotar for her several comments. The research was supported by the postdoctoral fellowship of Indian Institute of Technology (IIT) Kanpur. The author thanks the Institute for their support.

\end{document}